\documentclass[reqno,12pt,a4paper]{amsart}

\usepackage{array}
\usepackage[utf8]{inputenc}
\usepackage[english]{babel}
\usepackage{lmodern}
\usepackage{amssymb}

\usepackage[left=2.5cm,paperwidth=21cm,right=2.5cm,paperheight=29.7cm,textheight=23.5cm,top=2.5cm]{geometry}

\usepackage{tikz}
\usetikzlibrary{matrix,quotes,cd,decorations.markings}

\usepackage{enumitem} 
\usepackage{hyperref} 
\hypersetup{hidelinks,colorlinks,citecolor={blue!70!black},urlcolor=black,linkcolor={green!70!black}}

\usepackage[noabbrev]{cleveref} 

\usepackage{bookmark}

\title{Torsion order and irrationality of complete intersections}
\author{Jan Lange}
\address{Institute of Algebraic Geometry, Leibniz University Hannover, Welfengarten 1, 30167 Hannover, Germany.}
\email{\href{mailto:lange@math.uni-hannover.de}{lange@math.uni-hannover.de}}

\author{Guoyun Zhang}
\address{Shanghai Center for Mathematical Sciences, Fudan University, Songhu Road 2005, Shanghai, China.}
\email{\href{mailto:gyzhang21@m.fudan.edu.cn}{gyzhang21@m.fudan.edu.cn}}

\date{October 28, 2025}
\subjclass[2020]{primary 14M10, 14C25, secondary 14E08, 13C40} 
\keywords{torsion order, complete intersection, rationality problem}

\newtheorem{theorem}{Theorem}[section]
\AddToHook{env/theorem/begin}{\crefalias{section}{theorem}}
\newtheorem{proposition}[theorem]{Proposition}
\AddToHook{env/proposition/begin}{\crefalias{theorem}{proposition}}
\newtheorem{corollary}[theorem]{Corollary}
\AddToHook{env/corollary/begin}{\crefalias{theorem}{corollary}}
\newtheorem{lemma}[theorem]{Lemma}
\AddToHook{env/lemma/begin}{\crefalias{theorem}{lemma}}
\newtheorem*{claim}{Claim}

\theoremstyle{definition}
\newtheorem{definition}[theorem]{Definition}
\AddToHook{env/definition/begin}{\crefalias{theorem}{definition}}
\newtheorem{example}[theorem]{Example}
\AddToHook{env/example/begin}{\crefalias{theorem}{example}}
\newtheorem{remark}[theorem]{Remark}
\AddToHook{env/remark/begin}{\crefalias{theorem}{remark}}

\newcolumntype{C}{>{$}c<{$}}

\DeclareMathOperator{\CH}{CH}
\DeclareMathOperator{\charac}{char}

\DeclareMathOperator{\Grass}{Gr}
\DeclareMathOperator{\id}{id}
\DeclareMathOperator{\Ima}{Im}
\DeclareMathOperator{\LM}{LM}

\DeclareMathOperator{\Proj}{Proj}
\DeclareMathOperator{\Spec}{Spec}

\DeclareMathOperator{\Tor}{Tor}
\DeclareMathOperator{\trdeg}{trdeg}

\DeclareRobustCommand\longtwoheadrightarrow{\relbar\joinrel\twoheadrightarrow}

\newcommand{\restr}[2]{\left.{#1}\right|_{#2}}

\newcommand{\N}{\mathbb{N}}
\newcommand{\Z}{\mathbb{Z}}

\newcommand{\C}{\mathbb{C}}

\newcommand{\aff}{\mathbb{A}}
\newcommand{\CP}{\mathbb{P}}

\allowdisplaybreaks

\numberwithin{equation}{section}

\begin{document}

\begin{abstract}
    We provide new logarithmic lower bounds for the torsion order of a very general complete intersection in projective space as well as a very general hypersurface in products of projective spaces and Grassmannians, in particular we prove their retract irrationality.
\end{abstract}

\maketitle

\section{Introduction}

The torsion order $\Tor(X)$ of an integral variety $X$ over a field $k$ is the smallest positive integer $e \in \Z_{\geq 1}$ such that a decomposition
    \begin{equation}\label{eq:torsion-order-intro}
        e \cdot \Delta_X = z \times X + Z \in \CH_{\dim X}(X \times_k X)
    \end{equation}
exists, where $z \in \CH_0(X)$ is a zero-cycle and $Z$ is a cycle on $X \times_k X$ whose support does not dominate the second factor; we set $\Tor(X) = \infty$ if no such decomposition exists. 

The notion originates from work of Bloch \cite{blo79} (using an idea of Colliot-Th\'el\`ene) and Bloch--Srinivas \cite{BS83} and has been studied for instance in \cite{CL17,Kah17,Sch21-torsion,LS24}.
Slight variants of the torsion order have been considered in \cite{ACTP} and \cite{Voi17}.

A rationally (chain) connected smooth projective variety $X$ has finite torsion order and $X$ admits a decomposition of the diagonal (meaning $\Tor(X) = 1$) if $X$ is (retract) rational or $\aff^1$-connected.
Thus the non-rationality of smooth projective (rationally connected) varieties can be shown by proving that the torsion order is strictly larger than $1$ (or equivalently disproving a decomposition of the diagonal).
This approach to the rationality problem has been successfully started by Voisin \cite{Voi15}.
The rationality problem itself has been extensively studied for a very general member of many classes of algebraic varieties using various methods. We refer the reader to the recent surveys \cite{Debarre-survey,Sch-survey} for an overview of the methods and results.

Two other properties of a smooth projective integral variety are captured by its torsion order.
If $f \colon X \longrightarrow Y$ is a generically finite morphism between smooth projective integral varieties over a field $k$, then a ``pull-push-argument'' shows that $\Tor(Y)$ divides $\deg f \cdot \Tor(X)$. In particular, $\Tor(Y)$ provides a lower bound for the degree of any unirational parametrization of $Y$.
The torsion order of a smooth projective integral $k$-variety is also the smallest positive integer such that the kernel of the degree morphism $\deg \colon \CH_0(X_L) \longrightarrow \Z$ is $e$-torsion for any field extension $L/k$, where $\Tor(X) = \infty$ if and only if no such integer exists.

In this paper, we provide new and improved lower bounds for the torsion order of very general complete intersections in projective space as well as hypersurfaces in products of projective spaces and Grassmannians.

\subsection{Complete intersections}

An argument of Rojtman \cite{Roi80} for hypersurfaces extended to complete intersections by Chatzistamatiou--Levine \cite[Proposition 4.1]{CL17} shows that the torsion order of a Fano complete intersection of multidegree $(d_1,\dots,d_s)$ in $\CP^{N+s}$ over any field divides $d_1 ! \cdot \dots \cdot d_s!$, in particular it is finite.

In the same paper, Chatzistamatiou--Levine provide a lower bound for very general Fano complete intersections by building on earlier work of Totaro \cite{Tot16} and Koll\'ar \cite{Kol95a}.
The torsion order of a very general complex complete intersection $X$ of multidegree $(d_1,\dots,d_s)$ and dimension $N$ is divisible by some prime power $p^{m'}$ if
$$
    d_i \geq p^{m'} \left\lceil \frac{N + s + 1 - d_1 - \dots - d_s + d_i}{p^{m'} + 1} \right\rceil
$$
for some $i \in \{1,\dots,s\}$ if $p$ is odd or $N$ is even, see \cite[Theorem 7.2]{CL17}. By using the Fano index $r(X) := \dim X + s + 1 - d_1 - \dots - d_s$, the statement roughly means that $\Tor(X)$ is divisible by a positive integer $m$ if some degree $d_i$ satisfies
$$
    d_i \geq m \cdot r(X) + m.
$$
For hypersurfaces, this bound was significantly improved by Schreieder \cite{Sch19JAMS,Sch21-torsion}: A very general hypersurface $X$ of degree $d$ and Fano index $r(X) > 0$ over a field $k$ has torsion order divisible by a positive integer $m$, if $m$ is invertible in $k$ and $$
    d \geq \log_2(r(X)+m) + m.
$$
We extend the bound of Schreieder to complete intersections.

\begin{theorem}\label{thm:intro-torsion-order-ci}
    The torsion order of a very general complete intersection of multidegree $(d_1,\dots,d_s)$ and positive Fano index $r$ over a field $k$ is divisible by a positive integer $m$ if $m \in k^\ast$ and there exists a degree $d_i$ such that $d_i \geq \log_{2} (r + m) + m$.
\end{theorem}

A complete intersection of multidegree $(d_1,\dots,d_s)$ is \emph{very general} if it is (up to a field extension) isomorphic to the geometric generic fibre of the parameter space of complete intersection of the given multidegree, see \Cref{sec:very-general}.

In fact we obtain slightly better bounds, which recover the bounds for the torsion order of hypersurfaces in \cite{LS24}.
For instance, we obtain the following result for the $2$-divisibility of the torsion order, which yields the same bound in positive characteristic as the stable irrationality bound of Nicaise--Ottem \cite{NO22} in characteristic $0$. The results in \cite{NO22} are proven using a motivic method, which builds on earlier work of Nicaise--Shinder \cite{NS19} and Kontsevich--Tschinkel \cite{KT19}. The motivic method was recently improved by Hotchkiss--Stapleton \cite{HS25} to handle $\aff^1$-connectedness.

\begin{theorem}\label{thm:intro-ci-2-torsion}
    Let $k$ be a field of characteristic different from $2$. Then the torsion order of a very general complete intersection of multidegree $(d_1,\dots,d_s)$, dimension at least $4$ and Fano index $r$ is divisible by $2$ if $r \leq 2$ or some degree $d_i \geq 4$ and
    $$
        r \leq (d_i + 1) 2^{d_i - 4} - \left\lfloor\frac{d_i + 2}{2}\right\rfloor.
    $$
    In particular, such complete intersections are neither (retract) rational nor $\aff^1$-connected.
\end{theorem}

Note that the torsion order of complete intersections with non-positive Fano index $r$ might be infinite, e.g.\ non Fano complete intersection over $\C$ have infinite torsion order. The case of hypersurfaces in the above theorem is covered by \cite{Moe23}, who showed the stable irrationality in char $0$, and \cite[Theorem 1.1]{LS24}. The latter results also cover the case of quartic fivefolds, see also \cite{NO22} and \cite[Theorem 1.1]{PS23}.

\subsection{Hypersurfaces in rational varieties}

We illustrate the flexibility of our approach by considering also hypersurfaces in other rational varieties. Among them hypersurfaces in products of projective spaces are probably the most extensively studied ones with regard to the rationality problem, see for example \cite{BvB18,AbbanOkada,ABvBP20,NO22,Moe23}.

To the best of our knowledge, there seems to be no upper bound for the torsion order of hypersurfaces in products of projective spaces known in the literature. Although one might be able to adapt Rojtman's argument for complete intersection also to these examples, which we do not pursue here. Instead we provide the following lower bound.

\begin{theorem}\label{thm:intro-hyper-in-prod}
    Let $k$ be a field and let $m,s \in \Z_{\geq 1}$ be positive integers such that $m$ is invertible in $k$. Then the torsion order of a very general hypersurface in $\CP_k^{M_1} \times_k \dots \times_k \CP_k^{M_s}$ of multidegree $(d_1,d_2,\dots,d_s)$ is divisible by $m$, if $$
        M_1 \geq 4 \quad \text{and} \quad (d_1,d_2,\dots,d_s) \geq (\log_2(M_1) + m,M_2+1,M_3+1,\dots,M_s+1).
    $$
    In particular, it is neither stably rational nor retract rational nor $\aff^1$-connected.
\end{theorem}

This extends and generalizes earlier results by Nicaise--Ottem \cite[Proposition 6.2]{NO22}, where stable irrationality for $s = 2$ is shown in characteristic $0$.

A different class of examples are hypersurfaces in Grassmannians, which are related to Gushel-Mukai varieties, see e.g.\ \cite{DK18}.
We obtain the following result by relating certain hypersurfaces in Grassmannians directly with the ones in \cite{Sch19JAMS,LS24}.

\begin{theorem}\label{thm:intro-2-torsion-Grass}
    Let $l,n \in \Z_{\geq 1}$ be positive integers such that $l (n-l) \geq 4$. Consider the complex Grassmannian $\Grass(l,n)$ and fix a Plücker embedding $\Grass(l,n) \hookrightarrow \CP_{\C}^N$. Then the torsion order of the intersection of $\Grass(l,n)$ with a very general complex hypersurface of degree $d \geq 4$ in $\CP_{\C}^N$ is divisible by $2$, if $l(n-l) \leq (d+1) 2^{d-4}$. 
\end{theorem}

The restriction to the complex numbers is mostly for simplicity of the argument and it should be possible to obtain a similar result also in positive characteristic, see \Cref{rem:hypers-in-Grass}.

The stable rationality for hypersurfaces in the Grassmannian $\Grass(2,n)$ is studied in an unpublished work of Ottem\footnote{available at his webpage: \href{https://www.mn.uio.no/math/personer/vit/johnco/papers/gr25.pdf}{https://www.mn.uio.no/math/personer/vit/johnco/papers/gr25.pdf}.} and also in work of Yoshino \cite{Yos24,Yos25}.
For $l = 2$, Yoshino obtained a slightly better bound ($2(n-2) \leq (d+1)2^{d-4} + 2$) for the stable irrationality of hypersurfaces in $\Grass(2,n)$, see \cite[Theorem 1.3]{Yos25}.

\subsection{Outline of the argument}

Nicaise--Ottem \cite{NO22} proves successfully the (stable) irrationality of algebraic varieties using a degeneration to snc schemes, where the obstruction to rationality lies in some lower stratum using the motivic volume.
A cycle-theoretic analogue of this method was given first by Pavic--Schreieder \cite{PS23} for proper families with a generalization to non-proper families by Lange--Schreieder \cite{LS24}.

We elaborate the idea of \cite{LS24} to consider non-proper degenerations to snc schemes by studying affine degenerations to which the method of \cite{PS23,LS24} is applicable. 
In fact, we reformulate most of the assumptions in terms of basic commutative algebra statements.
An iterated application of this affine degeneration allows us to construct affine complete intersections from affine hypersurfaces, where the torsion order of the former is controlled by the torsion order of the latter.

Applying this machinery to the concrete hypersurface examples from \cite{Sch19JAMS,Sch21-torsion,LS24} provides affine complete intersections with a (non-trivial) lower bound on the torsion order. 
We emphasize that this ultimately relies on certain vanishing results for unramified cohomology classes found for instance in \cite{Sch19Duke,Sch19JAMS}.

The main results are obtained by considering different compactifications of these affine examples together with a standard degeneration argument.
To control the compactification, we use the theory of Gröbner basis.

We would like to point out that we use degenerations similar to and inspired by the ones in \cite{NO22,Moe23}. 
While they use one degeneration into many components, which they control via tropical geometry, we consider several simple degenerations into $2$ components.

In \Cref{sec:preliminaries}, we recall the theory of Gröbner basis, a relative version of the torsion order used in \cite{LS24}, and the notion of very general in families. 
The framework of affine degenerations is developed in \Cref{sec:affine-degenerations}. 
In \Cref{sec:base-examples}, we check that the hypersurfaces from \cite{Sch19JAMS,Sch21-torsion,LS24} as well as special complete intersections related to an example of Hassett--Pirutka--Tschinkel \cite{HPT18} fit in our framework of affine degeneration. 
This essentially boils down to some elementary computations using the explicit equations.
The final section (\Cref{sec:applications}) is devoted to the proof of the main results.

\section*{Acknowledgement}

The project started while GZ was visiting the Leibniz University Hannover with the support of the China Scholarship Council program (project ID:202406100211), and he is grateful to Stefan Schreieder for his valuable guidance during this period.
JL is supported by the Studienstiftung des deutschen Volkes.
We thank Finn Bartsch, V\'ictor Gonz\'alez Alonso, Robin Lahni, Ludvig Modin, Johannes Oertel, Simon Pietig, Stefan Schreieder, Isabel Vogt, and Lin Zhou for comments and discussions.

This project received funding from the European Research Council (ERC) under the European Union’s Horizon 2020 research and innovation programme under grant agreement No.~948066 (ERC-StG RationAlgic).

\section{Preliminaries}\label{sec:preliminaries}

\subsection{Notations and conventions}

All rings are assumed to be commutative with $1$. 
The characteristic $\charac(\Lambda)$ of a ring $\Lambda$ is the smallest positive integer $c \in \Z_{\geq 1}$ such that any element in $\Lambda$ is $c$-torsion and zero if no such positive integer exists.
The exponential characteristic of a field $k$ is $1$ if $\charac k = 0$ and $\charac k$ otherwise.
Unless stated otherwise, the generators of a polynomial ring are assumed to have degree $1$.

Let $R$ be a ring. By an $R$-scheme we mean a separated scheme of finite type over $R$. 
If $R = k$ is a field, then we also say \emph{algebraic scheme} over $k$.
An integral algebraic scheme over $k$ is called a \emph{$k$-variety} (or short: \emph{variety}).
If $X$ is an $R$-scheme and $A$ is an $R$-algebra, then $X \times_R A$ (or simply $X_A$) denotes the fibre product $X \times_{\Spec R} \Spec A$.

For an algebraic scheme $X$ over a field $k$, we denote its Chow group of dimension $i$ cycles by $\CH_i(X)$. We denote the tensor product $\CH_i(X) \otimes_\Z \Lambda$ for a ring $\Lambda$ by $\CH_i(X,\Lambda)$.

\subsection{Monomial orderings and Gröbner basis}

We recall a few basic notions related to monomial ordering of polynomial rings following the reference \cite[Section 2.2]{CLO15}.

Let $k$ be a field and let $R = k[x_1,\dots,x_n]$ be the polynomial ring over $k$ in $n$ variables. A \emph{monomial ordering} on $R$ is a well-ordering, i.e.~a total ordering such that any non-empty subset admits a smallest element, on the set of monomials $x^\alpha = x_1^{\alpha_1} \cdots x_n^{\alpha_n}$ where $\alpha = (\alpha_1,\dots,\alpha_n) \in \Z_{\geq 0}^n$ such that for every $\alpha,\beta,\gamma \in \Z_{\geq 0}^n$ $$
    x^\alpha > x^\beta \Longrightarrow x^{\alpha + \gamma} > x^{\beta + \gamma},
$$
see \cite[Section 2.2 Definition 1]{CLO15}.
A monomial ordering $>$ is called \emph{graded} if $\vert \alpha \vert > \vert \beta \vert$ implies $x^\alpha > x^\beta$ for every $\alpha, \beta \in \Z^n_{\geq 0}$, where $\vert \alpha \vert := \alpha_1 + \dots + \alpha_n$. The most important example for this paper is the graded lexicographical ordering.

\begin{example}[{\cite[Section 2.2 Definition 5]{CLO15}}]\label{example:grlex-ordering}
    Consider the relation $>_{\text{grlex}}$ on the set of monomials $x^\alpha$, $\alpha \in \Z_{\geq 0}^n$, on $R = k[x_1,\dots,x_n]$ given by $x^\alpha >_{\text{grlex}} x^\beta$ if $\vert \alpha \vert > \vert \beta \vert$ or $\vert \alpha \vert = \vert \beta \vert$ and the left-most non-zero entry in $\alpha-\beta \in \Z^n$ is positive. Then $>_{\text{grlex}}$ is a graded monomial ordering on $R$, called the \emph{graded lexicographic ordering}.
\end{example}

Let $R = k[x_1,\dots,x_n]$ be a polynomial ring over a field $k$ and fix a monomial ordering on $R$. For a polynomial $f = \sum_\alpha a_\alpha x^\alpha \in R$ we define the \emph{leading monomial} of $f$ as
\begin{equation}\label{eq:def-leadingmonomial}
    \LM(f) := \max \{x^\alpha : a_\alpha \neq 0\},
\end{equation}
where the maximum is taken with respect to the choosen monomial ordering. We say that the leading monomials of a collection $f_1,\dots,f_r$ of polynomials in $R$ are \emph{relatively prime} if for all $i,j \in \{1,\dots,r\}$ with $i \neq j$ we have
$$
    \gcd(\LM(f_i),\LM(f_j)) = 1.
$$

\begin{example}
    Consider the polynomials
    $$
        f = 3x^2y^3 + 6x^3y^2 - 5 xy + 5, \quad g = x^5 + x^2y^2 + y^3 + xy - 1, \quad h = y^2 + y + 1 \in k[x,y].
    $$
    The leading monomials with respect to the graded lexicographic ordering ($x > y$) are
    $$
        \LM(f) = x^3 y^2, \quad \LM(g) = x^5, \quad \LM(h) = y^2.
    $$
    In particular, we see that the leading monomials of $g$ and $h$ are relatively prime, while the leading monomials of $f$ and $g$ are not. 
\end{example}

\begin{definition}\label{def:homogenization}
Let $f\in k[x_1,...,x_n]$ be a polynomial of degree $d$. Then $$
    f^h(x_0,\dots,x_n) := x_0^d f\left(\frac{x_1}{x_0},...,\frac{x_n}{x_0}\right)
$$
is a homogeneous polynomial of degree $d$ in $k[x_0,...,x_n]$, called the \emph{homogenization} of $f$.

Let $I\subset k[x_1,...,x_n]$ be an ideal. Then we define the \emph{homogenization} of I to be the ideal 
$$
I^h= \left( f^h : f\in I \right) \subset k[x_0,x_1,\dots,x_n],
$$
which is a homogeneous ideal in $k[x_0,x_1,...,x_n]$.
\end{definition}

A collection of polynomials whose leading monomials are relatively prime in the above sense behave well under homogenization.

\begin{proposition}\label{prop:homogenization-monomial-order}
    Let $R = k[x_1,\dots,x_n]$ be a polynomial ring over a field $k$ and fix a graded monomial ordering on $R$.
    If the leading monomials of a collection of polynomials $f_1,\dots,f_r \in R$ of degree at least $1$ are relatively prime, then
    $
        I^h = (f_1^h,\dots,f_r^h).
    $
    In particular, the scheme $\{f_1^h = \dots = f_r^h = 0\} \subset \CP^n_k$ is the projective closure of $\{f_1 = \dots = f_r = 0\} \subset \aff^n_k$, if $k$ is algebraically closed.
\end{proposition}

\begin{proof}
    This is a consequence of the theory of Gröbner basis. The assumption on the relative primeness of the leading monomials implies that $f_1,\dots,f_r$ is a Gröbner basis for $I$ by \cite[Section 2.9, Proposition 4 and Theorem 3]{CLO15}. The homogenization of a Gröbner basis of $I$ is a basis of $I^h$, see e.g. \cite[Section 8.4, Theorem 4]{CLO15}.
    The claim about the projective closure is shown for example in \cite[Section 8.4, Theorem 8]{CLO15}.
\end{proof}

\subsection{Strictly semi-stable schemes}

We recall the well-known definition.

\begin{definition}\label{def:strictly-semi-stable}
    Let $R$ be a discrete valuation ring with residue field $k$ and fraction field $K$. A \emph{strictly semi-stable $R$-scheme} is an integral separated scheme $\mathcal{X}$ flat and of finite type over $R$ such that
\begin{itemize}
    \item the generic fibre $\mathcal{X}_K$ is smooth over $K$;
    \item the special fibre $Y := \mathcal{X}_k$ is geometrically reduced;
    \item each irreducible component $Y_i$ of $Y$ ($i \in I$) is a Cartier divisor in $\mathcal{X}$;
    \item for every $\emptyset \neq J \subseteq I$, the scheme-theoretic intersection $Y_J := \bigcap_{j \in J} Y_j$ is smooth and equidimensional of dimension $\dim Y + 1 - |J|$, if non-empty.
\end{itemize}
\end{definition}

\subsection{Relative torsion orders}\label{sec:relative-torsion-order}

Let $X$ be a variety over a field $k$. Let $\Delta_X \subset X \times_k X$ denote the diagonal class in the Chow group, and $\delta_X\in \CH_0(X_{k(X)})$ be the class induced by pulling back $\Delta_X$ via the natural morphism $X_{k(X)}\to X\times_k X$.

\begin{definition}[{\cite[Definition 3.3]{LS24}}]\label{def:relative-torsion-order}
    Let $\Lambda$ be a ring and let $X$ be a variety over a field $k$. Let $W \subset X$ be a closed subset and denote its complement by $U := X \setminus W$. 
    The \emph{$\Lambda$-torsion order of $X$ relative to $W$}, denoted by $\Tor^\Lambda(X,W)$, is the order of the element 
    $$
        \restr{\delta_X}{U} = \delta_U \in \CH_0(U_{k(X)},\Lambda).
    $$
\end{definition}

Note that $\Tor^\Lambda(X,W) \in \Z_{\geq 1} \cup \{\infty\}$ and it divides the characteristic of $\Lambda$ if $\charac(\Lambda) > 0$.
We recall a few basic facts about the relative torsion order, starting with the following simple lemma.

\begin{lemma}[{\cite[Lemma 3.6]{LS24}}]\label{lem:properties-of-torsion-order}
    Let X be a variety over a field $k$ and let $W\subset X$ be a closed subscheme. Then the following hold:
    \begin{enumerate}[label=(\arabic*)]
        \item For all $m\in \mathbb{Z}$, $\Tor^{\mathbb{Z}/m}(X,W)\mid \Tor^{\Z}(X,W)$;
        \item Let $W^{\prime} \subset W \subset X$, then $\Tor^{\Lambda}(X,W)\mid \Tor^{\Lambda}(X,W^{\prime})$; \label{item:lem_properties-tor-order:bigger-subset}
        \item $\Tor(X)$ is the minimum of the relative torsion orders $\Tor^{\mathbb{Z}}(X,W)$ where $W\subset X$ runs through all closed subsets of dimension zero;
        \item If $X$ is proper and $\deg \colon \CH_0(X)\cong \mathbb{Z}$ is an isomorphism, then $\Tor(X)=\Tor^\mathbb{Z}(X,W)$ for any closed subset $W\subset X$ of dimension $0$ containing a zero-cycle of degree $1$;
        \item $\Tor^{\Lambda}(X_L,W_L)\mid \Tor^{\Lambda}(X,W)$ for any ring $\Lambda$ and for any field extension $L/k$. Moreover, if $k=\bar{k}$ is algebraically closed, then $\Tor^{\Lambda}(X_L,W_L)= \Tor^{\Lambda}(X,W)$. \label{item:lem_properties-tor-order:field-extension}
    \end{enumerate}
\end{lemma}

\begin{lemma}\label{lem:torsion-order-and-CH0}
    Let $X$ be a smooth projective variety over an algebraically closed field $k$. If $H^{0}(X,\Omega_X^1) = 0$, then the torsion order $\Tor(X) = \infty$ or $\CH_0(X) \cong \Z$.
\end{lemma}

The lemma is a consequence of Rojtman's theorem and says that for smooth projective varieties over algebraically closed field with $h^{1,0} = 0$ the torsion order is infinite or it is the torsion order relative to any non-empty closed zero-dimensional subset.

\begin{proof}
    Let $X$ be a smooth projective variety over an algebraically closed field with Hodge number $h^{1,0} = 0$. Assume that the torsion order $\Tor(X)$ is finite. 
    Then it follows by an ``action of correspondence'' argument that the group of degree $0$ zero-cycles $\CH_0(X)_0$ is $l := \Tor(X)$-torsion.
    By Rojtman's theorem \cite{BlochRoitman,Roi80,Mil82}, the Albanese morphism induces an isomorphism
    $$
        \CH_0(X)_0 \longrightarrow \operatorname{Alb}(X)(l)
    $$
    between $\CH_0(X)_0$ and the $l$-torsion points of the Albanese variety of $X$. By \cite{Igusa-Alb}, we have $\dim \operatorname{Alb}(X) \leq h^{1,0}(X) = 0$. Hence $\CH_0(X)_0 = 0$, which proves the lemma.
\end{proof}

Recall that the relative torsion order behaves well under degeneration.

\begin{lemma}[{\cite[Lemma 3.8]{LS24}}]\label{lem:torsion-order-degeneration}
    Let $R$ be a dvr with fraction field $K$ and residue field $k$. Let $\mathcal{X}\rightarrow \Spec R$ be a separated flat $R$-scheme of finite type with geometrically integral fibres. 
    Denote the geometric generic fibre by $\bar{X} := \mathcal{X} \times_R \bar{K}$ and set $\bar{Y} := \mathcal{X} \times_R \bar{k}$.
    
    Then for any ring $\Lambda$ and for any closed subscheme $W_{\mathcal{X}} \subset \mathcal{X}$ such that the fibres of $\mathcal{X} \setminus W_{\mathcal{X}}$ are non-empty we have
    $$
        \Tor^\Lambda(\bar{Y},W_{\mathcal{X}} \times_{\mathcal{X}} \bar{Y})\mid \Tor^\Lambda (\bar{X},W_{\mathcal{X}} \times_{\mathcal{X}} \bar{X}).
    $$
\end{lemma}

The key technical result in \cite{LS24} relates the torsion order of the geometric generic fibre of a degeneration into two components with the torsion order of the intersection of the two components. We state a simplified version of \cite[Theorem 4.3]{LS24}.

\begin{theorem}\label{thm:degeneration-torsion-order}
    Let $R$ be a discrete valuation ring with algebraically closed residue field $k$ and fraction field $K$.
    Let $\Lambda$ be a ring of positive characteristic $c \in \Z_{\geq 1}$, i.e.~every $\lambda \in \Lambda$ is $c$-torsion, and assume that $c \in k^\ast$.
    Let $\mathcal{X}$ be a flat separated $R$-scheme of finite type with geometrically integral generic fibre $X = \mathcal{X}_K$ and special fibre $Y = \mathcal{X}_k$.
    Let $W_\mathcal{X} \subset \mathcal{X}$ be a closed subscheme and we denote by $W_V := W_\mathcal{X} \cap V$ the scheme-theoretic intersection with a subscheme $V \subset \mathcal{X}$. Assume the following:
    \begin{enumerate}[label=(\arabic*)]
        \item \label{item:thm_degeneration-torsion-order_2comp} $Y$ consists of two components $Y_0,Y_1$ whose intersection $Z := Y_0 \cap Y_1$ is integral; 
        \item \label{item:thm_degeneration-torsion-order_semistable} $\mathcal{X}^\circ := \mathcal{X} \setminus W_\mathcal{X}$ is a strictly semi-stable $R$-scheme, see \Cref{def:strictly-semi-stable}; 
        \item \label{item:thm_degeneration-torsion-order_rational} the variety $Y_i \setminus W_{Y_i}$ is isomorphic to an open subscheme of $\aff_k^{\dim Y_i}$ for $i = 0,1$.
    \end{enumerate}
    Then we have
    $$
        \Tor^\Lambda(Z,W_Z) \mid \Tor^\Lambda(X \times_K \bar{K},W_{X} \times_K \bar{K}).
    $$
\end{theorem}

\subsection{Very general elements of a family}\label{sec:very-general}

We recall the notion of a very general element in a family, as stated for instance in \cite[Section 2.5]{Sch-survey}.

\begin{definition}\label{def:very-general}
    Let $k$ be a field and let $\pi \colon \mathcal{X} \to B$ be a proper flat morphism of algebraic $k$-schemes with $B$ geometrically integral. We say that a closed point $b \in B$ is \emph{very general (with respect to $\pi$)} if there exist (algebraically closed) field extensions $K/\kappa(b)$ and $L/k(B)$ and an isomorphism $\varphi \colon K \to L$ of fields which induces an isomorphism
    $$
        X_b \times_{\kappa(b)} K \cong \mathcal{X} \times_{k(B)} L,
    $$
    where $X_b := \pi^{-1}(b)$ denotes the fibre over $b$. 
    
    If $B$ is a fine moduli space or a parameter space of some class of algebraic varieties and $\pi \colon \mathcal{X} \to B$ is its universal family, we often simply say that $X_b$ is very general if $b \in B$ is.
\end{definition}

\begin{remark}
    If $k$ is an uncountable algebraically closed field, then \cite[Lemma 2.1]{Vi13} shows that there exists a subset $U \subset B(k)$, which is the intersection of countably many non-empty open subsets of $B$, such that any point in $U$ is very general in the above sense.
\end{remark}

We say an algebraic scheme $X$ over a field $L$ \emph{degenerates} to an algebraic scheme $Y$ over an algebraically closed field $k$, if there exist a dvr $R$ with residue field $k$ and fraction field $K$, a flat proper morphism $\mathcal{X}\rightarrow \Spec R$, and an injection of field $K \hookrightarrow L$ such that $Y \cong \mathcal{X} \times_R k$ and $X \cong \mathcal{X} \times_R L$, see for example \cite[Section 2.6]{Sch19JAMS}.

\begin{lemma}\label{lem:very-general}
    Let $k$ be a field and let $\pi \colon \mathcal{X} \to B$ be a proper flat morphism of algebraic $k$-schemes such that $B$ is geometrically integral.
    \begin{enumerate}[label=(\roman*)]
        \item \label{item:very-gen:existence} If $\trdeg_{k_0} k \geq \dim B$, where $k_0 \subset k$ is a subfield over which $\pi$ is defined, then there exists a very general element $b \in B$.
        \item \label{item:very-gen:field-ext} Let $k'/k$ be a field extension.
        If $b \in B$ is very general with respect to $\pi$, then any closed point $b' \in B_{k'}$ lying over $b$ is very general with respect to $\pi_{k'} \colon X_{k'} \to B_{k'}$.
        \item \label{item:very-gen:degeneration} Up to a base-change, a very general fibre $X_b$ degenerates to the fibre $X_0$ over any closed point $0 \in B$ in the above sense.
    \end{enumerate}
\end{lemma}

Note that there exists a subfield $k_0 \subset k$, which is finitely generated over its prime field, such that $\pi$ is defined over $k_0$, as $\pi \colon \mathcal{X} \to B$ is a morphism of algebraic schemes over $k$. 
We provide a proof of this well-known lemma for the convenience of the reader.

\begin{proof}
    We start proving \ref{item:very-gen:existence}. Let $k_0 \subset k$ be a subfield such that $\pi$ is defined over $k_0$, i.e.~there exists a morphism $\pi_0 \colon \mathcal{X}_0 \longrightarrow B_0$ of algebraic schemes over $k_0$ such that all squares in the diagram
    $$
        \begin{tikzcd}
            \mathcal{X} \arrow[r,"\pi"] \arrow[d] & B \arrow[r] \arrow[d] & \Spec k \arrow[d] \\
            \mathcal{X}_0 \arrow[r,"\pi_0"] & B_0 \arrow[r] & \Spec k_0
        \end{tikzcd}
    $$
    are Cartesian. Assuming $\trdeg_{k_0} k \geq \dim B$, there exists a finite field extension $k^+/k$ and a homomorphisms of fields $k_0(B_0) \to k^+$.
    By the universal property of the fibre product, there exists a morphism $\iota \colon \Spec k^+ \to B$ of $k$-schemes such that the composition $\Spec k^+ \to B \to B_0$ factors through $\Spec k_0(B_0)$.
    In particular, $\iota$ determines a closed point $b \in B$ and we claim that $b$ is very general. 
    Indeed, note that $\mathcal{X} = \mathcal{X}_0 \times_{B_0} B$ and the morphisms $$
        \Spec k^+ \overset{\iota}{\longrightarrow} B \longrightarrow B_0 \quad \text{and} \quad \Spec k(B) \overset{\eta}{\longrightarrow} B \longrightarrow B_0,
    $$
    factors through $\Spec k_0(B_0)$, where $\eta$ is the inclusion of the generic point. Thus
    $$
        \begin{tikzcd}
        X_b \times_{\kappa(b)} k^+ \arrow[r] \arrow[d] & \mathcal{X}_0 \times_{B_0} \Spec k_0(B_0) \arrow[d] &  X_\eta \arrow[r] \arrow[d] & \mathcal{X}_0 \times_{B_0} \Spec k_0(B_0) \arrow[d] \\
        \Spec k^+ \arrow[r] & \Spec k_0(B_0) &  \Spec k(B) \arrow[r] & \Spec k_0(B_0).
        \end{tikzcd}
    $$
    are Cartesian squares.
    Here, $\kappa(b)$ denotes the residue field of the closed point $b$, which is a subfield of $k^+$ and $X_\eta$ denotes the generic fibre of the morphism $\pi \colon \mathcal{X} \to B$. In particular, it suffices to find (algebraically closed) field extensions $K/k^+$ and $L/k(B)$ together with an isomorphism of field $K \to L$ which fixes $k_0(B_0)$, see also the argument in \cite[Lemma 2.1]{Vi13}.
    The existence of such fields $K$ and $L$ is clear, as any two algebraically closed field extension of the same transcendence degree over the same field are abstractly isomorphic.

    By a similar argument as above, the statement \ref{item:very-gen:field-ext} reduces to the following question on fields. Given an isomorphism of algebraically closed fields $\varphi \colon K \to L$ with subfields $\kappa(b) \subset K$ and $k(B) \subset L$ as well as field extensions $\kappa(b')/\kappa(b)$ and $k'(B_{k'})/k(B)$, there exists algebraically closed fields $K'$ and $L'$ and an isomorphism $\varphi' \colon K' \to L'$ such that $K'$ contains $\kappa(b')$ and $K$ as subfields, $L'$ contains $k'(B_{k'})$ and $L$ as subfields, and $\varphi'$ extends $\varphi$. The existence follows by the same reasoning as in the proof of \ref{item:very-gen:existence}.

    We turn to the proof of \ref{item:very-gen:degeneration}. Let $0 \in B$ be any closed point. Up to replacing $k$ by an algebraically closed field extension whose transcendence degree over $k$ is at least $\dim B$, we can assume by \ref{item:very-gen:existence} that there exists a closed point $b \in B$, which is very general in the sense of \Cref{def:very-general}.
    It clearly suffices to show that the fibre $X_b := \pi^{-1}(b)$ degenerates to $X_0$.
    By a Bertini-type argument, there exists a geometrically integral curve $C \subset B$ containing $0$ and $b$, see e.g.~\cite[Corollary 1.9]{CP-Bertini}. 
    Up to replacing $C$ by its normalization, we can assume that there exists a smooth geometrically integral curve $C$ and a flat proper morphism $\pi_C \colon \mathcal{X}_C \to C$ and two closed points $c_0,c \in C$ 
    such that $c \in C$ is very general in the sense of \Cref{def:very-general} and the fibres over $c_0$ and $c$ are isomorphic to $X_0$ and $X_b$, respectively.
    Thus we see that $X_b$ degenerates to $X_0$, as $\mathcal{O}_{C,c_0}$ is a dvr.
\end{proof}

\begin{example}
    We will use the following two examples of families over a field $k$.
    \begin{enumerate}[label=(\alph*)]
        \item Complete intersections in $\CP^N$: Let $\mathbf{d} := (d_1,\dots,d_s) \in \Z^{s}_{\geq 1}$ be a collection of positive integers, where $1 \leq s < N$ is an integer. Then the affine scheme
        $$
            B_1 := H^0(\CP^N,\mathcal{O}_{\CP^N}(d_1)) \times_k \dots \times_k H^0(\CP^N,\mathcal{O}_{\CP^N}(d_s))
        $$
        parametrizes sets of $s$ polynomials $f_1,\dots,f_s \in k[x_0,\dots,x_N]$ with $\deg f_i = d_i$. 
        Let $\mathcal{X}_1 \subset B_1 \times_k \CP^N$ be the closed subscheme such that the fibre of $\mathcal{X}_1 \to B_1$ over a point corresponding to $f_1,\dots,f_s$ is the closed subscheme of $\CP^N$ cut out by the homogeneous ideal $(f_1,\dots,f_s)$.
        There exists an open subscheme $B \subset B_1$ such that the fibre over any point $b \in B$ has dimension $N - s$. Thus the morphism $\mathcal{X} := \mathcal{X}_1 \times_{B_1} B \to B$ is a family of complete intersections of multidegree $\mathbf{d}$.

        Other choices for $B$ include (an open subscheme of) a product of Grassmannians or an iteratated Grassmannian bundle over $\Spec k$. We refer the reader to \cite[\S 2.2]{Ben12} and \cite[\S 1.2]{DL23} for more details on moduli spaces for complete intersections.
        \item The parameter space of hypersurfaces of multidegree $\mathbf{d} := (d_1,\dots,d_s)$ in the product of projective spaces $\CP^{M_1}_k \times_k \dots \times_k \CP^{M_s}_k$ is given by $$
            B := \CP\left(H^0\left(\CP^{M_1},\mathcal{O}_{\CP^{M_1}}(d_1)\right) \otimes_k \dots \otimes_k H^0\left(\CP^{M_s},\mathcal{O}_{\CP^{M_s}}(d_s)\right)\right).
        $$
        Then we consider the universal family $\mathcal{X} \to B$ of multidegree $\mathbf{d}$ hypersurfaces in $\CP^{M_1}_k \times_k \dots \times_k \CP^{M_s}_k$.
    \end{enumerate}
\end{example}

\section{Affine degenerations}\label{sec:affine-degenerations}

In this section, we construct (strictly semi-stable) families over a dvr, which satisfy the assumptions of \Cref{thm:degeneration-torsion-order}. In particular, we need to write down a family whose special fibre consists of two irreducible components and whose total space is regular. This is straightforward in the affine setting.
Let $A$ be a smooth $k$-algebra of finite type and let $f_1,f_2 \in A$, then the natural morphism \begin{equation}\label{eq:naive-degeneration}
    \Spec \left(A \otimes_k k[t]_{(t)}/(t-f_1f_2)\right) \longrightarrow \Spec k[t]_{(t)}
\end{equation}
is a degeneration such that the total space is regular and the special fibre is equal to
$$
    \Spec A/(f_1f_2) \cong \Spec \left( A/(f_1) \right) \cup \Spec \left(A/(f_2)\right).
$$
A key assumption in \Cref{thm:degeneration-torsion-order} is the ``rationality'' condition \ref{item:thm_degeneration-torsion-order_rational}, which we reformulate to an algebraic condition for degenerations as \eqref{eq:naive-degeneration} in the first part.
In the second part, we provide two general constructions of families as in \eqref{eq:naive-degeneration} which satisfy the assumptions in \Cref{thm:degeneration-torsion-order}.

\subsection{Strongly rational algebras}

The assumption \ref{item:thm_degeneration-torsion-order_rational} in \Cref{thm:degeneration-torsion-order} states that the irreducible components of the special fibre are isomorphic to an open subscheme of affine space. Rephrasing this geometric condition algebraically leads to the following definition.

\begin{definition}\label{def:strongly-rational}
    Let $k$ be a field. A finite type $k$-algebra $B$ is called \emph{strongly $k$-rational} if $B$ is isomorphic as a $k$-algebra to the localization of a polynomial ring over $k$ and $B \neq 0$.
\end{definition}

\begin{example}\label{ex:strongly-rational-algebra}
    \begin{enumerate}[label=(\arabic*)]
        \item The polynomial ring $k[x_1,\dots,x_n]$ is clearly a strongly $k$-rational algebra for every $n \geq 0$.
        \item \label{item:strongly-rational-algebra:1+xy} The finite type $k$-algebra $ k[x_1,x_2]/(1 + x_1x_2)$ is strongly $k$-rational, as $$
            k[x_1,x_2]/(1 + x_1x_2) \cong k[x_2,x_2^{-1}].
        $$
        \item\label{item:strongly-rational-algebra:field-ext} If $B$ is a strongly $k$-rational algebra, then $B \otimes_k L$ is strongly $L$-rational for every field extension $L/k$.
    \end{enumerate}
\end{example}

The following remark verifies that the spectrum of a strongly rational $k$-algebra is isomorphic to an open in affine space, in particular rational.

\begin{remark}\label{rem:strongly-rational-algebra}
    Let $B$ be a strongly $k$-rational algebra.
    Then there exists a polynomial ring $A = k[x_1,\dots,x_n]$ and a multiplicative subset $S \subset A$ such that $B \cong S^{-1} A$.
    Since $B$ is also of finite type over $k$ by the definition of strongly rational $k$-algebra, we find that $\Spec B \to \Spec A = \aff^n_k$ is an open immersion by \cite[Theorem 2.7]{Oda04}, see also the Added part in \textit{loc.\ cit.} for an alternative argument\footnote{We found the argument and reference through an \href{https://mathoverflow.net/a/20792}{answer of Martin Brandenburg on mathoverflow}}.
\end{remark}

The above remark implies that strongly $k$-rational algebras are geometrically integral and smooth $k$-algebras. We provide an algebraic proof for the convenience of the reader.

\begin{lemma}\label{lem:strongly-rat-implies-geom-int-smooth}
    A strongly $k$-rational algebra $B$ is a geometrically integral, smooth $k$-algebra.
\end{lemma}

\begin{proof}
    It suffices to prove that $B$ is an integral domain and (formally) smooth over $k$, which is obvious as $B$ is by definition isomorphic as a $k$-algebra to the localization of a polynomial ring over $k$.
\end{proof}

\subsection{Affine strictly semi-stable families}

We provide a general setup to construct affine strictly semi-stable schemes over the dvr $k[t]_{(t)}$ such that its special fibre consists of two irreducible components which are rational.

\begin{definition}\label{def:admissible-pair}
    Let $k$ be a field and let $A$ be an integral smooth $k$-algebra of finite type and of Krull dimension $N+1$. We say that two elements $f_1 \in A[z] := A \otimes_k k[z]$ and $f_2 \in A$ are \emph{admissible with respect to $A$} if the quotient ring $A[z]/(f_1,f_2)$ is a geometrically integral $k$-algebra of dimension $N$ and the $k$-algebras $A[z]/(f_1)$ and $A/(f_2)$ are strongly $k$-rational and of Krull dimension $N+1$ and $N$, respectively.
\end{definition}

The above definition is motivated by the following proposition.

\begin{proposition}\label{prop:general-deg-strictly-semi-stable}
    Let $k$ be a field and let $A$ be an integral smooth $k$-algebra of finite type and of Krull dimension $N+1$. Assume that $f_1 \in A[z]$ and $f_2 \in A$ are admissible with respect to $A$ in the sense of \Cref{def:admissible-pair}. Then the affine scheme
    \begin{equation}\label{eq:general-strictly-semi-stable}
        \Spec (A[z]_{\partial_z f_1} \otimes_k R)/(t-f_1f_2)
    \end{equation}
    is a strictly semi-stable $R := k[t]_{(t)}$-scheme.
    Moreover, the $k(t)$-algebra \begin{equation}\label{eq:general-strictly-semi-stable-closure-generic-fibre}
        A[z] \otimes_k k(t)/(t-f_1f_2)
    \end{equation}
    is an integral domain.
\end{proposition}

\begin{proof}
    We check the definition of strictly semi-stable $R$-schemes, see \Cref{def:strictly-semi-stable}.
    Note first that the $k[t]$-algebra $$
        B := A[z,t]/(t - f_1f_2) \cong A[z]
    $$ 
    is a smooth $k$-algebra and thus in particular flat. As $k[t]$ is an unramified $k$-algebra, $B$ is a flat $k[t]$-algebra by \cite[Remark 18.30]{GW-AG2}. Thus the $R$-algebra
    $$
        B \otimes_{k[t]} R = (A[z] \otimes_k R)/(t-f_1f_2) 
    $$
    is flat. Localizing at the multiplicative subset $\{(\partial_z f_1)^n : n \in \N\} \subset B \otimes_{k[t]} R$, then implies that \eqref{eq:general-strictly-semi-stable} is a flat $R$-scheme.
    Since $B$ is clearly an integral domain and being an integral domain is preserved under localization, it follows that the $k(t)$-algebra \eqref{eq:general-strictly-semi-stable-closure-generic-fibre} is an integral domain and the affine scheme \eqref{eq:general-strictly-semi-stable} is integral.
    
    The generic fibre of the $R$-scheme \eqref{eq:general-strictly-semi-stable} is the spectrum of the $k(t)$-algebra
    \begin{equation}\label{eq:general-strictly-semi-stable-generic-fibre}
        (A \otimes_k k(t))[z,w]/(t-f_1f_2,w \partial_z f_1 - 1),
    \end{equation}
    which is a standard smooth $A \otimes_k k(t)$-algebra in the sense of \cite[\href{https://stacks.math.columbia.edu/tag/00T6}{Tag 00T6}]{stacks-project}. Thus \eqref{eq:general-strictly-semi-stable-generic-fibre} is a smooth $k(t)$-algebra, as $A$ is a smooth $k$-algebra by assumption.

    The special fibre of the $R$-scheme \eqref{eq:general-strictly-semi-stable} has two irreducible components, $\Spec A[z]_{\partial_z f_1}/(f_1)$ and $\Spec A[z]_{\partial_z f_1}/(f_2)$, which are both smooth $k$-schemes by \Cref{lem:strongly-rat-implies-geom-int-smooth} and clearly Cartier divisors in \eqref{eq:general-strictly-semi-stable}. The intersection is the spectrum of the $k$-algebra
    $$
        A[z]_{\partial_z f_1}/(f_1,f_2) \cong \left(A/(f_2)\right)[z,w]/(f_1,w \partial_z f_1 - 1),
    $$
    which is a smooth $k$-algebra by a similar argument as for \eqref{eq:general-strictly-semi-stable-generic-fibre}.
    Note that this uses the fact that $A/(f_2)$ is strongly $k$-rational and thus a smooth $k$-algebra by \Cref{lem:strongly-rat-implies-geom-int-smooth}. 
\end{proof}

The following example provides a way to construct admissible pairs with respect to strongly $k$-rational algebras.

\begin{example}\label{ex:adding-hyperplane}
    Let $k$ be a field and let $B$ be a strongly $k$-rational algebra, e.g.\ a polynomial ring. Let $f \in B[z]$ such that $B[z]/(f)$ is geometrically integral of Krull dimension $\dim B$. Then the pair $f_1 = f + y \in B[y,z]$ and $f_2 = y \in B[y]$ is admissible with respect to $B[y]$. Indeed, the assumptions on $B$ and $f$ together with the $k$-algebra isomorphisms
    $$
        B[y,z]/(f_1,f_2) \cong B[z]/(f), \quad B[y,z]/(f_1) = B[y,z]/(f + y) \cong B[z], \quad B[y]/(f_2) = B
    $$
    imply directly that $f_1$ and $f_2$ are admissible with respect to $B[y]$.
\end{example}

\subsection{The key construction}

In this part, we provide the key construction of a new admissible pair from an admissible pair. 
The idea is the following simple observation: Let $f_1 \in A[z]$ and $f_2 \in A$ be admissible with respect to a strongly $k$-rational algebra $A$ for some field $k$. Then the generic fibre of the family \eqref{eq:naive-degeneration} is the spectrum of the $k(t)$-algebra
\begin{equation}\label{eq:key-construction-intro-generic-fibre}
    A[z] \otimes_k k(t)/(t-f_1f_2) \cong A[z,y] \otimes_k k(t)/(t+f_1y,f_2+y).
\end{equation}
Note that the $k(t)$-algebras
$$
    A[z,y] \otimes_k k(t)/(t+f_1y) \cong \left(A[z] \otimes_k k(t)\right)_{f_1} \quad \text{and} \quad A[z,y] \otimes_k k(t)/(f_2+y) \cong A[z] \otimes_k k(t)
$$
are strongly $k$-rational. Thus $g_1 := t + f_1 y$ and $g_2 := f_2 + y$ are admissible with respect to $A[z,y] \otimes_k k(t)$ if the $k(t)$-algebra \eqref{eq:key-construction-intro-generic-fibre} is geometrically integral.

As integral varieties over a field $K$ with a $K$-rational point in the smooth locus are geometrically integral, see e.g.\ \cite[\href{https://stacks.math.columbia.edu/tag/0CDW}{Tag 0CDW}]{stacks-project}, we introduce the following extra condition for $f_1 \in A[z]$ and $f_2 \in A$
\begin{equation}\label{eq:condition-star}
    \forall F/k \quad \forall (q_1,q_2) \in F^2 \quad \exists \ \text{F-algebra epimorphism} \ A[z]_{\partial_z f_1} \otimes_k F/(f_1 + q_1,f_2 + q_2) \twoheadrightarrow F. \tag{$\star$}
\end{equation}

The following lemma shows that this extra condition guarantees that the $k(t)$-algebra \eqref{eq:key-construction-intro-generic-fibre} admits a $k(t)$-rational point in the smooth locus.

\begin{lemma}\label{lem:rational-pt+geom-int}
    Let $k$ be a field and let $A$ be an integral smooth $k$-algebra of finite type. Assume that $f_1 \in A[z]$ and $f_2 \in A$ are admissible with respect to $A$ in the sense of \Cref{def:admissible-pair} and satisfy condition \eqref{eq:condition-star}. Then the $k(t)$-scheme
    $$
        \Spec (A[z]_{\partial_z f_1} \otimes_k k(t))/(t-f_1f_2)
    $$
    contains a $k(t)$-rational point and is geometrically integral over $k(t)$.
\end{lemma}

\begin{proof}
    The existence of a $k(t)$-rational point is equivalent to showing that the $k(t)$-algebra $$
        (A[z]_{\partial_z f_1} \otimes_k k(t))/(t+f_1y,f_2 + y)
    $$
    admits a $k(t)$-algebra epimorphism to $k(t)$. The latter can be constructed as
    $$
        (A[z,y]_{\partial_z f_1} \otimes_k k(t))/(t+f_1y,f_2 + y) \twoheadrightarrow (A[z]_{\partial_z f_1} \otimes_k k(t))/(t+f_1,f_2 + 1) \twoheadrightarrow k(t),
    $$
    where the first surjection is the quotient by the ideal $(y-1)$ and the second surjection exists by condition \eqref{eq:condition-star} applied to the field extension $F = k(t)$ and the pair $(q_1,q_2) = (t,1)$.

    Thus the smooth and integral $k(t)$-variety $$
        \Spec (A[z]_{\partial_z f_1} \otimes_k k(t))/(t-f_1f_2),
    $$ 
    see \Cref{prop:general-deg-strictly-semi-stable}, is geometrically integral over $k(t)$ by \cite[\href{https://stacks.math.columbia.edu/tag/0CDW}{Tag 0CDW}]{stacks-project}.
\end{proof}

\begin{corollary}\label{cor:new-admissible-pair}
    Let $k$ be a field and let $A$ be a strongly $k$-rational algebra.
    If $f_1 \in A[z]$ and $f_2 \in A$ are admissible with respect to $A$ in the sense of \Cref{def:admissible-pair} and satisfy the condition \eqref{eq:condition-star}, then the two pairs $(h_1,h_2)$ and $(\tilde{h}_1,\tilde{h}_2)$ with
    $$
        \begin{aligned}
            &h_1 := t + f_1 y \in A[y,z] \otimes_k k(t) \quad \text{and} \quad h_2 := f_2 + y \in A[y] \otimes_k k(t) \\
            &\tilde{h}_1 := f_1 + y \in A[y,z] \otimes_k k(t) \quad \text{and} \quad \tilde{h}_2 := t + f_2 y \in A[y] \otimes_k k(t)
        \end{aligned}
    $$
    are admissible with respect to $A[y] \otimes_k k(t)$ and also satisfy \eqref{eq:condition-star}.
\end{corollary}

\begin{proof}
    We check the definition of an admissible pair: Note that $$
        A[y,z] \otimes_k k(t)/(h_1,h_2) \cong A[z] \otimes_k k(t)/(t-f_1f_2) \cong A[y,z] \otimes_k k(t)/(\tilde{h}_1,\tilde{h}_2)
    $$
    Thus the quotient is geometrically integral over $k(t)$ by \Cref{lem:rational-pt+geom-int}. 
    Recall that for $i = 1,2$ the $k(t)$-algebras
    $$
        A[z,y] \otimes_k k(t)/(t+f_i y) \cong \left(A[z] \otimes_k k(t)\right)_{f_i} \quad \text{and} \quad A[z,y] \otimes_k k(t)/(f_i+y) \cong A[z] \otimes_k k(t)
    $$
    are strongly $k(t)$-rational. Thus the pairs $(h_1,h_2)$ and $(\tilde{h}_1,\tilde{h}_2)$ are admissible with respect to $A[y] \otimes_k k(t)$.
    To show condition \eqref{eq:condition-star}, let $F/k(t)$ be a field extension and let $(q_1,q_2) \in F^2$. Then the composition
    $$
        A[y,z]_{\partial_{z} g_1} \otimes_k F/(h_1+q_1,h_2+q_2) \longrightarrow A[z]_{\partial_{z} f_1} \otimes_k F/(t + f_1 + q_1,f_2 + 1 + q_2) \xrightarrow{\eqref{eq:condition-star}} F
    $$
    yields an $F$-algebra epimorphism, where the first map is the quotient by the ideal $(y - 1)$ and the second map exists because $f_1$ and $f_2$ satisfy the condition \eqref{eq:condition-star}. Hence, the pair $(h_1,h_2)$ satisfies condition \eqref{eq:condition-star}. The same argument works for $(\tilde{h}_1,\tilde{h}_2)$.
\end{proof}

We end this part by reformulating the condition \eqref{eq:condition-star} for the construction in \Cref{ex:adding-hyperplane}.

\begin{example}\label{ex:adding-hyperplane2}
    Let $k$ be a field and let $B$ be a strongly $k$-rational algebra, e.g.\ a polynomial ring. Let $f \in B[z]$ be neither a zero-divisor nor a unit. Assume that $B[z]/(f)$ is geometrically integral and that $f$ satisfies the condition
    \begin{equation}\label{eq:condition-star2}
        \forall F/k \ \forall q \in F \ \exists \ \text{F-algebra epimorphism} \ B[z]_{\partial_z f} \otimes_k F/(f + q) \twoheadrightarrow F. \tag{$\star \star$}
    \end{equation}
    Then the pair $f_1 = f + y \in B[y,z]$ and $f_2 = y \in B[y]$ is admissible with respect to $B[y]$ by \Cref{ex:adding-hyperplane} and satisfies condition \eqref{eq:condition-star}, as
    $$
        B[y,z]_{\partial_z (f + y)} \otimes_k F/(f + y + q_1,y + q_2) \cong B[z]_{\partial_z f} \otimes_k F/(f -q_2 + q_1)
    $$
    for every field extension $F/k$ and $q_1,q_2 \in F$.
\end{example}

\subsection{Admissible pairs and torsion orders}

In this part, we apply \Cref{thm:degeneration-torsion-order} to the degeneration \eqref{eq:general-strictly-semi-stable}, which provides us a flexible way to understand the torsion order of complete intersections by understanding the torsion order of hypersurfaces.

The first theorem is a key technical result, which controls the torsion order in affine degeneration of the form \eqref{eq:general-strictly-semi-stable} and can be seen as an algebraic version of \Cref{thm:degeneration-torsion-order}.

\begin{theorem}\label{thm:induction-step}
    Let $k$ be an algebraically closed field and let $\Lambda$ be a commutative ring with $1$ and of finite characteristic $e \in \Z_{\geq 1}$ such that $e \in k^\ast$. Let $A$ be an integral smooth $k$-algebra of finite type. Assume that $f_1 \in A[z]$ and $f_2 \in A$ are admissible with respect to $A$ and satisfy condition \eqref{eq:condition-star}. Let $l \in A[z]$ be a non-trivial element and consider
    $$
        \begin{aligned}
            &W_Z := \Spec A[z]/(f_1,f_2,l \cdot \partial_z f_1) \subset Z := \Spec A[z]/(f_1,f_2) \\
            &W_{\bar{X}} := \Spec A[z] \otimes_k \overline{k(t)}/(t-f_1 f_2,l \cdot \partial_z f_1) \subset \bar{X} := \Spec A[z] \otimes_k \overline{k(t)}/(t-f_1 f_2),
        \end{aligned}
    $$
    where $\overline{k(t)}$ is an algebraic closure of a purely transcendental extension $k(t)/k$.
    Then $\bar{X}$ is an integral variety and $\Tor^\Lambda(Z,W_Z) \mid \Tor^\Lambda(\bar{X},W_{\bar{X}})$.
\end{theorem}

\begin{proof}
    Let $R := k[t]_{(t)}$ and consider the $R$-scheme
    \begin{equation}\label{eq:induction-step-degeneration}
        \mathcal{X} := \Spec \left( A[z] \otimes_k R/(t-f_1 f_2) \right).
    \end{equation}
    Note that $Z$ is the intersection of the two irreducible components of the special fibre and $\bar{X}$ is the geometric generic fibre of $\mathcal{X}$. We aim to apply \Cref{thm:degeneration-torsion-order} to the family \eqref{eq:induction-step-degeneration}; thus we check that the assumptions in \Cref{thm:degeneration-torsion-order} are satisfied.

    The family \eqref{eq:induction-step-degeneration} is flat and separated over $R$ by \cite[Proposition 14.20 and Proposition 9.15]{GW-AG1}. It is clearly of finite type.

    The generic fibre $X$ is geometrically integral, as it is integral by \Cref{prop:general-deg-strictly-semi-stable}, Cohen-Macaulay by \cite[Proposition 18.13]{Eis95}, and contains an open (dense) subscheme, which is geometrically integral by \Cref{lem:rational-pt+geom-int}.

    Consider the closed affine subscheme $$
        W_\mathcal{X} := \Spec \left( A[z] \otimes_k R/(t-f_1 f_2,l \cdot \partial_z f_1) \right) \subset \mathcal{X}
    $$
    Note that $W_\mathcal{X} \times_{\mathcal{X}} Z = W_Z$ and $W_\mathcal{X} \times_{\mathcal{X}} \bar{X} = W_{\bar{X}}$ for $W_Z$ and $W_{\bar{X}}$ as in the statement of the theorem. The open subscheme $\mathcal{X}^\circ := \mathcal{X} \setminus W_{\mathcal{X}} \subset \mathcal{X}$ is given by
    $$
        \Spec \left(A[z] \otimes_k R/(t-f_1 f_2) \right)_{l \cdot \partial_z f_1} = \Spec A[z]_{l\cdot \partial_z f_1} \otimes_k R/(t-f_1 f_2),
    $$
    where we used that $l, f_1 \in A[z]$. 
    Thus $\mathcal{X}^\circ$ is a strictly semi-stable $R$-scheme by \Cref{prop:general-deg-strictly-semi-stable} and so condition \ref{item:thm_degeneration-torsion-order_semistable} of \Cref{thm:degeneration-torsion-order} is satisfied.
    The special fibre of the family \eqref{eq:induction-step-degeneration} consists of the two irreducible varieties $$
        Y_0 \cong \Spec A[z]/(f_1) \quad \text{and} \quad Y_1 \cong A[z]/(f_2).
    $$
    By definition of an admissible pair, the intersection $Z = Y_0 \cap Y_1 \cong \Spec A[z]/(f_1,f_2)$ is (geometrically) integral and the varieties $Y_0$ and $Y_1$ are isomorphic to an open subscheme of affine space, see \Cref{rem:strongly-rational-algebra}. 
    Hence condition \ref{item:thm_degeneration-torsion-order_2comp} and \ref{item:thm_degeneration-torsion-order_rational} hold.

    In total, we have shown that the assumptions of \Cref{thm:degeneration-torsion-order} are satisfied for the family $\mathcal{X}$ and the closed subscheme $W_{\mathcal{X}}$. Thus the theorem follows from \Cref{thm:degeneration-torsion-order}.
\end{proof}

The following theorem provides a framework to inductively construct certain affine complete intersection from affine hypersurfaces by simultaneously controlling the (relative) torsion order. As the torsion order is related to the rationality, see the discussion in the introduction, the proposition can be seen as an affine analogue and generalization of the result \cite[Theorem 7.7]{NO22}, where stable irrationality is studied.

\begin{theorem}\label{thm:add-hypers}
    Let $k$ be an uncountable algebraically closed field and let $m \geq 2$ and $n,r \geq 0$ be integers with $m \in k^\ast$.
    If there exist non-constant polynomials
    $$
        f_1,\dots,f_r \in k[x_1,\dots,x_{n+r}] \quad \text{and} \quad f \in k[x_1,\dots,x_{n+r},z]
    $$ such that
    \begin{enumerate}[label=(C\arabic*)]
        \item \label{item:thm_add-hypers:strongly-rat} the $k$-algebra $B := k[x_1,\dots,x_{n+r}]/(f_1,\dots,f_r)$ is strongly $k$-rational;
        \item \label{item:thm_add-hypers:integral} the affine scheme $\Spec B[z]/(f)$ is an integral $k$-variety of dimension $\dim B$;
        \item \label{item:thm_add-hypers:tor-order}  there exists a non-zero polynomial $l \in k[x_1,\dots,x_{n+r},z]$ such that $$
            \Tor^{\Z/m}\left(\Spec B[z]/(f),\Spec B[z]/(f, l \cdot \partial_{z} f)\right) = m;
        $$
        \item \label{item:thm_add-hypers:ratl-pt}  $f \in B[z]$ satisfies condition \labelcref{eq:condition-star2}.
    \end{enumerate}
    Then for every integers $d$ and $M$ with $d \geq M \geq 1$, there exist polynomials $$
        \check{f},\tilde{f} \in k[x_1,\dots,x_{n+r},w_1,\dots,w_M,z] \quad \text{and} \quad f_{r+1} \in k[w_1,\dots,w_M]
    $$
    of degree $\deg \check{f} = \deg f + d$, $\deg \tilde{f} = \deg f$ and $\deg f_{r+1} = d$ such that
    \begin{enumerate}[label=(\alph*)]
        \item\label{item:thm_add-hypers:add-hypers} the polynomials $f_1,\dots,f_{r+1}$ and $\tilde{f}$ satisfy the properties \labelcref{item:thm_add-hypers:strongly-rat,item:thm_add-hypers:integral,item:thm_add-hypers:ratl-pt,item:thm_add-hypers:tor-order} if additionally $d = M$ or $M \geq 2$;
        \item\label{item:thm_add-hypers:increase-deg} the polynomials $f_1,\dots,f_{r}$ and $\check{f}$ satisfy the properties \labelcref{item:thm_add-hypers:strongly-rat,item:thm_add-hypers:integral,item:thm_add-hypers:ratl-pt,item:thm_add-hypers:tor-order} and the degrees with respect to the $x,z$- and $w$-coordinates are $\deg_{x,z} \check{f} = \deg f$ and $\deg_w \check{f} = d+1$.
    \end{enumerate}
\end{theorem}

We call $d$ the \emph{added degree} and $M$ the \emph{number of added variables}.

\begin{proof}
    The theorem is a consequence of the following claim, which we prove by applying inductively the results from \Cref{sec:affine-degenerations} together with a simple degeneration argument.
    \begin{claim}
        For all integers $d$ and $M$ with $d = M \geq 1$ or $d > M \geq 2$, there exist polynomials $$
            \tilde{f} \in k[x_1,\dots,x_{n+r},z,w_1,\dots,w_M] \quad \text{and} \quad f_{r+1} \in k[w_1,\dots,w_M]
        $$
        of degree $\deg f$ and $d$ respectively, such that $\tilde{f}$ and $f_{r+1}$ are admissible with respect to $B[w_1,\dots,w_d]$ for $B$ as in \ref{item:thm_add-hypers:strongly-rat} and satisfy condition \eqref{eq:condition-star} from \Cref{sec:affine-degenerations} as well as 
        $$
            \Tor^{\Z/m}\left(\Spec B[z,w_1,\dots,w_M]/(\tilde{f},f_{r+1}),\Spec B[z,w_1,\dots,w_M]/(\tilde{f},f_{r+1},l \cdot \partial_z \tilde{f})\right) = m
        $$
        for some $l \in k[x_1,\dots,x_{n+r},z]$. Additionally, we have $\tilde{f} - f \in k[w_1,\dots,w_M]$ of degree $1$.
    \end{claim}
    Note that the claim immediately implies \ref{item:thm_add-hypers:add-hypers}, see \Cref{def:admissible-pair} and \eqref{eq:condition-star}. The explicit construction of $\tilde{f}$ and $f_{r+1}$ in the claim allows us to construct $\check{f}$ and deduce \ref{item:thm_add-hypers:increase-deg}.
    
    The proof of the claim is divided into two steps. In the first step we show the claim for $d = M \geq 1$ via the results from \Cref{sec:affine-degenerations}. The second step proves the claim for $d > M \geq 2$ by applying a simple degeneration argument together with \Cref{lem:torsion-order-degeneration}.
    
    \textbf{Step 1.} Assume that $d = M \geq 1$. We construct inductively $\tilde{f}$ and $f_{r+1}$.
    If $d = 1$, then \Cref{ex:adding-hyperplane2} applied to the strongly $k$-rational algebra $B$ and $f \in B[z]$ shows that
    $$
        \tilde{f} = f + w \in k[x_1,\dots,x_{n+r},w,z] \quad \text{and} \quad f_{r+1} = w \in k[w]
    $$
    are admissible with respect to $B[w]$ and satisfy the condition \eqref{eq:condition-star}. Note that $$
        B[w,z]/(\tilde{f},f_{r+1}) = B[z]/(f),
    $$
    which shows the claim for $d = 1$.

    Suppose that for some $d \geq 1$ there exists polynomials $\tilde{f} \in k[x_1,\dots,x_{n+r},w_1,\dots,w_{d},z]$ and $f_{r+1} \in k[w_1,\dots,w_d]$ as in the claim with $M=d$.
    Consider a purely transcendental field extension $k(t)/k$. 
    Then the polynomials $$
    \begin{aligned}
        \tilde{g} &:= \tilde{f} + w_{d+1} \in k(t)[x_1,\dots,x_{n+r},w_1,\dots,w_{d+1},z] \quad \text{and} \quad \\ g_{r+1} &:= t + w_{d+1} f_{r+1} \in k(t)[w_1,\dots,w_{d+1}]
    \end{aligned}
    $$
    are admissible with respect to $B[w_1,\dots,w_{d+1}]$ and satisfy condition \eqref{eq:condition-star} by \Cref{cor:new-admissible-pair}.
    Note that
    $$
        B' := B \otimes_k k(t)[w_1,\dots,w_{d+1},z]/(\tilde{g},g_{r+1}) \cong B \otimes_k k(t)[w_1,\dots,w_d,z]/(t - \tilde{f} f_{r+1}).
    $$
    Thus, \Cref{thm:induction-step} and the assumption on the relative torsion order in the claim imply 
    $$
        \Tor^{\Z/m}\left(\Spec B' \otimes_{k(t)} \overline{k(t)}, \Spec (B' \otimes_{k(t)} \overline{k(t)})/ (l \cdot \partial_{z} \tilde{g})\right) = m.
    $$
    Up to identifying $\overline{k(t)}$ with $k$ via a field isomorphism, which exists as both fields are algebraically closed field extensions of the same (uncountable) transcendence degree over the same prime field, we can assume that the polynomials $\tilde{g}$ and $g_{r+1}$ are defined over $k$. 
    Note that $\tilde{g}$ and $g_{r+1}$ are admissible with respect to $B[w_1,\dots,w_{d+1}]$ and satisfy condition \eqref{eq:condition-star}. We also have $\deg \tilde{g} = \deg \tilde{f}$, $\deg g_{r+1} = d+1$, and $\deg_w \tilde{g} = 1$ as well as $$
        \Tor^{\Z/m}\left(\Spec B[w_1,\dots,w_{d+1},z]/(\tilde{g},g_{r+1}),\Spec B[w_1,\dots,w_{d+1},z]/(\tilde{g},g_{r+1}, l \cdot \partial_{z} \tilde{g})\right) = m.
    $$
    Thus, we have constructed inductively polynomials as in the claim for $d = M$, which concludes the first step.

    \textbf{Step 2.}
    Assume that $2 \leq M < d$. By Step 1, there exists polynomials
    $$
        \tilde{f} \in k[x_1,\dots,x_{n+r},w_1,\dots,w_M,z] \quad \text{and} \quad f_{r+1} \in k[w_1,\dots,w_M]
    $$
    of degree $\deg \tilde{f} = \deg f$ and $\deg f_{r+1} = M$, which satisfy the conditions in the claim for $d = M$. In fact, it follows from the construction in Step 1 that we can take
    \begin{equation}\label{eq:thm_add-hypers:Step2}
        \begin{aligned}
        \tilde{f} &= f + w_1 + \dots + w_M \in k[x_1,\dots,x_M,w_1,\dots,w_M,z] \, \text{and} \\ f_{r+1} &= t_2 - w_M (t_1 - w_{M-1} g) \in k[w_1,\dots,w_M]
        \end{aligned}
    \end{equation}
    for some $t_1 \in k$ with $t_1 = 0$ if and only if $M =2$, some $t_2 \in k^\ast$ which is transcendental over the prime field of $k$, and some polynomial $g \in k[w_1,\dots,w_{M-2}]$ of degree $M-2$.
    Let $\overline{k(t)}$ be an algebraic closure of a purely transcendental field extension $k(t)/k$.
    Consider the polynomials
    \begin{equation}\label{eq:thm_add-hypers:Step2-degeneration}
        \begin{aligned}
        \tilde{g} &:= \tilde{f} \in \overline{k(t)}[x_1,\dots,x_N,y_1,\dots,y_M] \quad \text{and} \\
        g_{r+1} &:= f_{r+1} + t w_M(w_{M-1} - \delta_{M,2}) h \in \overline{k(t)}[w_1,\dots,w_M],
        \end{aligned}
    \end{equation}
    where the polynomials $\tilde{f}$ and $f_{r+1}$ are as in \eqref{eq:thm_add-hypers:Step2}, $h \in k[w_1,\dots,w_{M-1}]$ is an arbitrary polynomial of degree $d-2$, and $\delta_{M,2}$ is the Kronecker delta, i.e. $\delta_{M,2} = 1$ if $M = 2$ and $\delta_{M,2} = 0$ otherwise.

    We check that the polynomials $\tilde{g}$ and $g_{r+1}$ in \eqref{eq:thm_add-hypers:Step2-degeneration} satisfy the conditions in the claim by using some abstract field isomorphism $k \cong \overline{k(t)}$, see also Step 1.

    The admissibility of $\tilde{g}$ and $g_{r+1}$ with respect to $B \otimes_k \overline{k(t)}[w_1,\dots,w_M]$ and condition \eqref{eq:condition-star} can be shown via similar arguments as in \Cref{sec:affine-degenerations}.
    We provide some details for the convenience of the reader. 
    Identify $\overline{k(t)}$ with $\overline{k(t,t_2)}$ and consider
    $$
        A := B \otimes_k k(t,t_2)[w_1,\dots,w_M].
    $$
    Note that $A$ is strongly $k(t,t_2)$-rational, as $B$ is strongly $k$-rational by \ref{item:thm_add-hypers:strongly-rat}, see \Cref{ex:strongly-rational-algebra} \ref{item:strongly-rational-algebra:field-ext}; in particular $A$ is a smooth geometrically integral $k(t,t_2)$-algebra of finite type by \Cref{lem:strongly-rat-implies-geom-int-smooth}.
    Viewing $\tilde{g}$ and $g_{r+1}$ from \eqref{eq:thm_add-hypers:Step2-degeneration} as elements in $A[z]$, we find that
    $$
        \begin{aligned}
             A[z]/(\tilde{g}) &\cong B[w_1,\dots,w_d,z]/(\tilde{f}) \otimes_k k(t,t_2) \\
             A/(g_{r+1}) &\overset{\eqref{eq:thm_add-hypers:Step2-degeneration}}{=} A/\left(t_2 - w_M (t_1 - w_{M-1} g - t(w_{M-1} - \delta_{M,2})h)\right)
        \end{aligned}
    $$
    are strongly $k(t,t_2)$-rational by the definition of admissible (\Cref{def:admissible-pair}) for $\tilde{f}$ and $f_{r+1}$
    and \Cref{ex:strongly-rational-algebra} \ref{item:strongly-rational-algebra:1+xy}.
    The elements $\tilde{g} \in A[z]$ and $g_{r+1} \in A$ satisfy condition \eqref{eq:condition-star}, as for every field extension $F/k(t,t_2)$ and any $q_1,q_2 \in F$, we have an $F$-algebra epimorphism
    \begin{equation}\label{eq:thm_add-hypers:Step2-rationalpt-epi}
            \frac{A[z]_{\partial_{z} \tilde{g}} \otimes_{k(t,t_2)} F}{(\tilde{g} + q_1, g_{r+1} + q_2)} \longtwoheadrightarrow 
            \frac{B \otimes_k F[w_1,\dots,w_{M-2},z]_{\partial_{z} f}}{\left(f + w_1 + \dots + w_{M-2} + \delta_{M,2} + \frac{t_2 + q_2}{t_1 - \delta_{M,2}g} + q_1\right)} \overset{\ref{item:thm_add-hypers:ratl-pt}}{\longtwoheadrightarrow} F.
    \end{equation}
    The first arrow is the quotient by the ideal $(w_{M-1}-\delta_{M,2})$ together with an isomorphism of $F$-algebras using the explict description of $\tilde{g}$ and $g_{r+1}$ from \eqref{eq:thm_add-hypers:Step2-degeneration}. This uses that $t_1 - \delta_{M,2}g \in k^\ast$.
    In particular, we see that the $k(t,t_2)$-scheme
    \begin{equation}\label{eq:thm_add-hypers:Step2-opensubscheme}
        \Spec A[z]_{\partial_{z} \tilde{g}}/(\tilde{g},g_{r+1}) = \Spec \left(A/(g_{r+1})\right)[z]_{\partial_{z} \tilde{g}}/(\tilde{g})
    \end{equation}
    contains a $k(t,t_2)$-rational point and is smooth over $k(t,t_2)$, as it is a standard smooth $\Spec A/(g_{r+1})$-scheme and the $k(t,t_2)$-algebra $A/(g_{r+1})$ is strongly $k(t,t_2)$-rational by the above, see also \Cref{lem:strongly-rat-implies-geom-int-smooth}.
    The scheme \eqref{eq:thm_add-hypers:Step2-opensubscheme} is an open subscheme of $\Spec A[z]/(\tilde{g},g_{r+1})$, see e.g.\ the argument in \Cref{rem:strongly-rational-algebra}.
    Both schemes are integral $k(t,t_2)$-varieties, as $A[z]/(\tilde{g},g_{r+1})$ is an localization of the integral domain
    \begin{equation}\label{eq:thm_add-hypers:Step2-iso-for-integral}
       B \otimes_k k(t)[t_2,w_1,\dots,w_M,z]/(\tilde{g},g_{r+1})
       \overset{\eqref{eq:thm_add-hypers:Step2-degeneration}}{\cong} B \otimes_k k(t)[w_1,\dots,w_M,z]/(\tilde{f}).
    \end{equation}
    Thus, the $k(t,t_2)$-scheme \eqref{eq:thm_add-hypers:Step2-opensubscheme} is geometrically integral by \cite[\href{https://stacks.math.columbia.edu/tag/0CDW}{Tag 0CDW}]{stacks-project}.
    Since \eqref{eq:thm_add-hypers:Step2-opensubscheme} is an open subscheme of the $k(t,t_2)$-variety $\Spec A[z]/(\tilde{g},g_{r+1})$, which is also Cohen-Macaulay by \cite[Proposition 18.13]{Eis95}, the variety $\Spec A[z]/(\tilde{g},g_{r+1})$ is also geometrically integral over $k(t,t_2)$.
    Hence $\tilde{g}$ and $g_{r+1}$ are admissible with respect to $A$, in particular also with respect to $B \otimes_k \overline{k(t)}[w_1,\dots,w_M]$.

    The statement about the torsion order follows now directly from \Cref{lem:torsion-order-degeneration} via the degeneration $t \to 0$.
    Hence, the polynomials $\tilde{g}$ and $g_{r+1}$ satisfy the conditions in the claim for $d > M \geq 2$, which concludes the second step and the proof of the claim.

    We turn to the construction of $\check{f}$ and the proof of \ref{item:thm_add-hypers:increase-deg}.
    Let $d$ and $M$ be positive integers such that $d \geq M$. Then, the claim applied for the integers $d+1$ and $M+1$ shows that there exists polynomials $$
        \tilde{f} \in k[x_1,\dots,x_{n+r},w_1,\dots,w_{M+1},z] \quad \text{and} \quad f_{r+1} \in k[w_1,\dots,w_{M+1}]
    $$
    of degree $\deg f$ and $d+1$, respectively,
    such that the conditions \labelcref{item:thm_add-hypers:ratl-pt,item:thm_add-hypers:tor-order,item:thm_add-hypers:integral,item:thm_add-hypers:strongly-rat} are satisfied for $f_1,\dots,f_{r+1}$ and $\tilde{f}$. In fact, the construction in Step 2 shows that the polynomials
    $$
        \begin{aligned}
            \tilde{f} &= f + w_1 + \dots + w_{M+1} \in k[x_1,\dots,x_{n+r},w_1,\dots,w_{M+1},z] \, \text{and} \\
            f_{r+1} &= t_2 - w_{M+1} \left(t_1 - w_{M} g - t(w_{M} - \delta_{M+1,2})h\right) \in k[w_1,\dots,w_{M+1}]
        \end{aligned}
    $$
    work for some specific $t,t_2 \in k^\ast$ and $t_1 \in k$ and some polynomials $g \in k[w_1,\dots,w_{M-1}]$ and $h \in k[w_1,\dots,w_{M}]$ of degree $M-1$ and $d-1$, respectively, see \eqref{eq:thm_add-hypers:Step2} and \eqref{eq:thm_add-hypers:Step2-degeneration}.
    We note that there is a $B$-algebra isomorphism
    \begin{equation}\label{eq:thm_add-hypers:B-algebra-iso-check-f}
        B[w_1,\dots,w_{M+1},z]/(\tilde{f},f_{r+1}) \cong B[w_1,\dots,w_M,z]/(\check{f}),
    \end{equation}
    where $\check{f}$ is the degree $(d + \deg f)$ polynomial
    \begin{equation}\label{eq:thm_add-hypers:def-check-f}
        \check{f} := t_2 + (f + w_1 + \dots + w_{M}) (t_1 - w_{M} g - t(w_{M} - \delta_{M+1,2})h)
    \end{equation}
    in $k[x_1,\dots,x_{n+r},w_1,\dots,w_M,z]$.
    It is straightforward to see from the isomorphism \eqref{eq:thm_add-hypers:B-algebra-iso-check-f} and the explicit form of $\check{f}$ in \eqref{eq:thm_add-hypers:def-check-f} that \ref{item:thm_add-hypers:increase-deg} holds for $\check{f}$ as in \eqref{eq:thm_add-hypers:def-check-f}.
\end{proof}

\begin{remark}\label{rem:to-add-hypers}
    \begin{enumerate}[label=(\alph*)]
        \item The proof shows that the polynomial $l$ in \ref{item:thm_add-hypers:tor-order} for the polynomials $f_1,\dots,f_r$ and $f$ also works for $f_1,\dots,f_{r+1}$ and $\tilde{f}$ as well as for $f_1,\dots,f_r$ and $\check{f}$.
        \item \label{item:rem:add-hypers_tildef} The proof shows that the polynomial $\tilde{f}$ can be choosen as $\tilde{f} = f + h$ for some linear polynomial $h \in k[w_1,\dots,w_M]$, see the statement of the claim in the proof.
        \item \label{item:rem:add-hypers_alternative-for-M2} If $M = 2$, then we can also choose the polynomials in \eqref{eq:thm_add-hypers:Step2-degeneration} as $$
            \tilde{g} := \tilde{f} = f + w_1 + w_2 \quad \text{and} \quad g_{r+1} := t_2 - w_1w_2 + tw_1^{d}.
            $$
            Indeed, the proof uses the strongly $k(t,t_2)$-rationality of $A/(g_{r+1})$, the existence of the epimorphism \eqref{eq:thm_add-hypers:Step2-rationalpt-epi}, and the isomorphism \eqref{eq:thm_add-hypers:Step2-iso-for-integral} as well as that $g_{r+1}$ specializes to $f_{r+1}$ from \eqref{eq:thm_add-hypers:Step2} via $t \to 0$. The last two properties are obviously still satisfied for the new choice of $g_{r+1}$. We can replace the epimorphism \eqref{eq:thm_add-hypers:Step2-rationalpt-epi} by the morphism
            $$
                A[z]_{\partial_{z} \tilde{g}} \otimes_{k(t,t_2)} F/(\tilde{g} + q_1, g_{r+1} + q_2) \longtwoheadrightarrow 
            B \otimes_k F[z]_{\partial_{z} f}/\left(f + t_2 + t + q_2 + q_1 + 1 \right) \longtwoheadrightarrow F
            $$
            where the first map is the quotient by the ideal $(w_1 - 1)$ and the second map exists by \ref{item:thm_add-hypers:ratl-pt}. The strongly $k(t,t_2)$-rationality of $A/(g_{r+1})$ follows directly from
            $$
                A/(g_{r+1}) \cong B \otimes_k k(t,t_2)[w_1,w_2]/(t_2 - w_1(w_2 + tw_1^{d -1})) \cong B \otimes_k k(t,t_2)[w_1]_{w_1}.
            $$
    \end{enumerate}
\end{remark}

\section{Base examples}\label{sec:base-examples}

To apply the machinery of affine degenerations from \Cref{sec:affine-degenerations}, we need explicit equations of (rationally connected) hypersurfaces with no decomposition of the diagonal.

\subsection{Schreieder's hypersurface examples} \label{sec:hypersurface-examples}

In \cite{Sch19JAMS,Sch21-torsion}, the author constructs examples of singular hypersurfaces with non-trivial unramified cohomology classes, in particular these examples have no decomposition of the diagonal.
Starting from these explicit hypersurfaces, Schreieder and the first named author \cite{LS24} manage to increase the dimension of the hypersurface examples via a degeneration argument, see \Cref{thm:degeneration-torsion-order}.
In this subsection, we recall these constructions and show that they fit into our affine degeneration framework, leading to the following result.

\begin{theorem}\label{thm:base-examples-hypersurfaces}
    Let $k$ be an algebraically closed and uncountable field. Let $n', m \geq 2$ be integers such that $m$ is invertible in $k$. Then for every integer $N$ satisfying
    $$
        3 \leq N \leq n' + 2^{n'} - 2 + \sum_{l=0}^{n'-1}\binom{n'}{l}\left\lfloor\frac{l}{m}\right\rfloor
    $$
    and any integer $d \geq n' + m$, there exists an irreducible polynomial 
    $$
        f \in k[x_1,\ldots,x_{N},z]
    $$
    of degree $d$ satisfying condition \eqref{eq:condition-star2} and such that 
    \begin{equation}\label{eq:base-example-tor-condition}
        \Tor^{\Z/m}\left(\Spec k[x_1,\dots,x_N,z]/(f), \Spec k[x_1,\dots,x_N,z]/(f,\partial_z f)\right) = m.
    \end{equation}
    In particular, $f$ satisfies the assumption
    \labelcref{item:thm_add-hypers:strongly-rat,item:thm_add-hypers:integral,item:thm_add-hypers:ratl-pt,item:thm_add-hypers:tor-order} in \Cref{thm:add-hypers} for $n = N$, $r = 0$.
\end{theorem}

We start with the singular hypersurface examples in \cite{Sch21-torsion}, see also the discussion in \cite[Section 6]{LS24}.

\begin{example}\label{ex:Sch-hyper-example}
    Let $k$ be an algebraically closed field containing an element $\pi \in k$, which is transcendental over the prime field of $k$.
    Fix integers $m,n \geq 2$ such that $m$ is invertible in $k$ and consider the polynomial
    \begin{equation}\label{eq:Sch-example-g-affine}
        g := \pi\cdot\left(1 + \sum_{i=1}^n x_i^{\lceil\frac{n+1}{m}\rceil}\right)^m-(-1)^n x_1x_2\dots x_n \in k[x_1,\dots,x_n].
    \end{equation}
    For an integer $N$ satisfying $n + 1 \leq N \leq n + 2^n - 2$, we consider the polynomial
    \begin{equation}\label{eq:Sch-example-f0-affine}
        f_0 := g(x_1,\dots,x_n) + \sum_{j=1}^{N - n} c_j(x_1,\dots,x_n) x_{n+j}^m + (-1)^nx_1x_2\dots x_n z^m \in k[x_1,\dots,x_N,z],
    \end{equation}
    where $g$ is as in \eqref{eq:Sch-example-g-affine} and we set $c_j := (-x_1)^{\varepsilon_1}(-x_2)^{\varepsilon_2}\dots(-x_n)^{\varepsilon_n} \in k[x_1,\dots,x_n]$ for the unique $\varepsilon_i \in \{0,1\}$ satisfying $j = \sum_{i=1}^{n} \varepsilon_i 2^{i-1}$.
    We note that $\deg f_0 = n + m$ and that $c_1(x_1,\dots,x_n) = x_1$.

    In fact, $f_0$ is the dehomogenization of the homogeneous polynomial $F$ appearing in \cite[Equation (6.3)]{LS24}. Thus, \cite[Theorem 6.1]{LS24} implies that $f_0$ is an irreducible polynomial and
    $$
        \Tor^{\Z/m}\left(\Spec k[x_1,\dots,x_N,z]/(f_0),\Spec k[x_1,\dots,x_N,z]/(f_0,\partial_z f_0)\right) = m,
    $$
    as $\partial_z f_0 = (-1)^n m x_1 \cdots x_n z^{m-1}$. 
    In particular, $f_0$ satisfies the conditions \labelcref{item:thm_add-hypers:integral,item:thm_add-hypers:strongly-rat,item:thm_add-hypers:tor-order} in \Cref{thm:add-hypers} for $n = N$ and $r = 0$.
\end{example}

The following lemma shows that $f_0$ also satisfies the assumption \ref{item:thm_add-hypers:ratl-pt}.

\begin{lemma}\label{lem:Sch-example-rationalpt}
    Let the notation be as in \Cref{ex:Sch-hyper-example}. Then for any field extension $F/k$ and any element $q \in F$, there exists an $F$-algebra epimorphism
    \begin{equation}\label{eq:Sch-example-rationalpt}
        F[x_1,\dots,x_N,z,w]/(f_0 + q, (-1)^n m x_1 \cdots x_n z^{m-1}w - 1) \longtwoheadrightarrow F.
    \end{equation}
\end{lemma}

\begin{proof}
    Let $F/k$ be a field extension and let $q \in F$.
    It clearly suffices to construct the epimorphism in the case that $N = n+1$. 
    If $q \in F$ is defined over $k$, then the existence of an epimorphism \eqref{eq:Sch-example-rationalpt} is straightforward, as $k$ is algebraically closed. Indeed, there exists clearly an $k$-algebra epimorphism
    $$
        k[x_1,\dots,x_{n+1},z,w]/(f_0 + q, w \partial_z f_0 - 1) \longtwoheadrightarrow k[x_{n+1}]/(x_{n+1}^m + \pi (n+1)^m + q) \longtwoheadrightarrow k,
    $$
    where the first morphism is the quotient by the ideal $(x_1 - 1,\dots,x_n -1,z-1)$ and the second homomorphism exists as $k$ is algebraically closed.
    Since tensor product is right-exact, the $F$-algebra surjection \eqref{eq:Sch-example-rationalpt} exists for $q \in k$.

    Thus, we may assume now that $q \in F \setminus k$. Let $\zeta \in k$ be such that $\zeta^{{\lceil\frac{n+1}{m}\rceil}} = -1$, which exists as $k$ is algebraically closed.
    Then the quotient by the ideal $$
        (x_2 - \zeta x_1, x_3 -1, \dots,x_{n+1}-1,z-1) \subset F[x_1,\dots,x_{n+1},z,w]
    $$
    yields an $F$-algebra homomorphism
    $$
        F[x_1,\dots,x_{n+1},z]_{\partial_z f_0}/(f_0 + q) \longtwoheadrightarrow F[x_1]_{x_1}/(\pi(n-1)^m + x_1 + q) \cong F,
    $$
    where we used for the isomorphism that $\pi (n-1)^m + q \neq 0$, as $q \notin k$.
\end{proof}

Starting from the above examples, Schreieder and the first named author construct a series of hypersurface examples without a decomposition of the diagonal via the ``double cone construction'', see \cite{Moe23} and \cite[Section 5]{LS24}, which we briefly recall here.

\begin{example}\label{ex:double-cone-construction}
    Let $k$ be an algebraically closed field and let $N,m,d \geq 2$ be integers such that $m$ is invertible in $k$ and $2m \leq d$.
    Suppose there exists an irreducible polynomial $f \in k[x_1,\dots,x_{N},z]$ of the form
    \begin{equation}\label{eq:LS-double-cone-form}
        f := b z^m + \sum\limits_{i=0}^m a_i x_{j_0}^i
    \end{equation}
    for some index $j_0 \in \{1,\dots,N\}$ and some polynomials $b,a_0,\dots,a_m \in k[x_1,\dots,\widehat{x_{j_0}},\dots,x_{N}]$ not containing the variable $x_{j_0}$ and of degree $\deg b = d - m$, $\deg a_i \leq d - i$  such that
    $$
        \Tor^{\Z/m}\left(\Spec k[x_1,\dots,x_N,z]/(f), \Spec k[x_1,\dots,x_{N},z]/(f,\partial_z f)\right) = m.
    $$
    For example, the polynomial $f_0$ in \eqref{eq:Sch-example-f0-affine} satisfies the assumptions for $j_0 = n+1$.
    
    Let $K/k$ be an algebraically closed field extension of transcendence degree $2$ and let $\lambda,t$ denote a transcendental basis of $K/k$.
    Consider the $K$-scheme
    $$
        X := \{f(x_1,\dots,x_N,z) + w_0 + (\lambda x_{j_0} + 1) w_1 = t - w_0 w_1 = 0 \} \subset \aff^{N+3}_{K},
    $$
    where $x_1,\dots,x_N,z,w_0,w_1$ are the coordinates of $\aff^{N+3}_{K}$. Note that $X$ is as a $K$-scheme isomorphic to
    $$
    \begin{aligned}
        X &\cong \Spec K[x_1,\dots,x_N,z,w_0,w_0^{-1}]/(f(x_1,\dots,x_N,z) + w_0 + (\lambda x_{j_0} + 1) tw_0^{-1}) \\ &\cong \Spec K[x_1,\dots,x_N,z,w_0,w_0^{-1}]/\left(b z^m + \sum\limits_{i=0}^m a_i (w_0 x_{j_0} - \lambda^{-1})^i +  w_0 + t \lambda x_{j_0}\right),
    \end{aligned}
    $$
    where the second isomorphism is given by the transformation $x_{j_0} \mapsto (w_0 x_{j_0} - \lambda^{-1})$.
    
    In particular, if $\deg a_i \leq d - 2i$ for all $i \in \{0,\dots,m\}$, then $X$ is isomorphic to an open subscheme of an affine degree $d$ hypersurface in $\aff^{N+2}_K$, namely the one associated to the irreducible polynomial
    \begin{equation}\label{eq:LS-double-cone-explicit}
        \tilde{f} := b z^m + \sum\limits_{i=0}^m a_i (w_0 x_{j_0} - \lambda^{-1})^i +  w_0 + t \lambda x_{j_0} \in K[x_1,\dots,x_N,z,w_0].
    \end{equation}
    By degenerating $X$ via $t \to 0$, it follows from \cite[Proposition 5.9]{LS24} that
    $$
        \Tor^{\Z/m}\left(\Spec K[x_1,\dots,x_N,z,w_0]/(\tilde{f}), \Spec K[x_1,\dots,x_N,z,w_0]/(\tilde{f},w_0 \cdot \partial_z \tilde{f})\right) = m,
    $$
    see also \cite[Proof of Theorem 7.1]{LS24}. Note that $\tilde{f}$ is again of the form \eqref{eq:LS-double-cone-form}.

    Observe that $\tilde{f}$ automatically satisfies the condition \eqref{eq:condition-star2}. Indeed, since $b$ is a non-zero polynomial over an infinite field $k$, there exists $\alpha := (\alpha_i)_{i = 1,\dots,\hat{j_0},\dots,N} \in k^{N-1}$ such that $b(\alpha) \neq 0$. Then the quotient by the ideal
    $$
        (z-1,x_{j_0},x_i-\alpha_i : i \neq j_0)
    $$
    induces for any field extension $F/K$ and any $q \in F$ the following $F$-algebra epimorphism
    $$
        F[x_1,\dots,x_N,z,w_0,w]/(\tilde{f} + q,w \partial_z \tilde{f} - 1) \twoheadrightarrow F[w_0]/\left(f(\alpha_1,\dots,\alpha_N,1) + w_0 + q\right) \cong F,
    $$
    where we set $\alpha_{j_0} := - \lambda$.
\end{example}

Repeatedly applying the ``double cone construction'' to the polynomial $f_0$ in \eqref{eq:Sch-example-f0-affine} yields polynomials of degree $\deg f_0$ in a polynomial ring of larger Krull dimension.

\begin{remark}\label{rem:double-cone-construction}
    Let $n,m \geq 2$ be integers such that $n,m \geq 2$. Then for $N = n + 2^n -2$ the polynomial $f_0$ from \Cref{eq:Sch-example-f0-affine} reads
    \begin{equation}\label{eq:Sch-hyper-example-max}
        f_0 = g(x_1,\dots,x_n) + \sum_{j=1}^{2^n-2} c_j(x_1,\dots,x_n) x_{n+j}^m + (-1)^nx_1x_2\dots x_n z^m
    \end{equation}
    in $k[x_1,\dots,x_{n + 2^n - 2},z]$,
    where $g,c_1,\dots,c_{2^n-2} \in k[x_1,\dots,x_n]$ are some specific non-zero polynomials, see \Cref{ex:Sch-hyper-example}.
    In particular, we see that the ``double cone construction'' is applicable for $j_0 \in \{n+1,n+2,\dots,n+2^n-2\}$.
    
    A quick analysis of the change of the polynomials $a_i$ under the ``double cone construction'' $f \to \tilde{f}$ shows that we can apply the construction from \Cref{ex:double-cone-construction} in total $$
        \sum\limits_{j=1}^{2^n-2} \left\lfloor \frac{n - \deg c_j}{m} \right\rfloor = \sum\limits_{l = 0}^{n-1} \binom{n}{l} \left\lfloor \frac{l}{m} \right\rfloor
    $$
    times, see \cite[Proof of Theorem 7.1 and also Lemma 5.3]{LS24}. 
\end{remark}

    We summarize the results of the above examples (\Cref{ex:Sch-hyper-example} and \Cref{ex:double-cone-construction}) before turning to the proof of the \Cref{thm:base-examples-hypersurfaces}.
    Let $k$ be an algebraically closed and uncountable field and let $n,m \geq 2$ be integers such that $m$ is invertible in $k^\ast$.
    For every integer $N$ with $n + 1 \leq N \leq n + 2^{n} - 2$, there exists an irreducible polynomial 
    \begin{equation}\label{eq:summary-Sch-examples}
        f \in k[x_1,\dots,x_{N},z]
    \end{equation} 
    of degree $n + m$ satisfying the conditions \eqref{eq:condition-star2} and \eqref{eq:base-example-tor-condition}.

\begin{proof}[Proof of \Cref{thm:base-examples-hypersurfaces}:]
    The construction of the polynomial $f$ involves the existence of a polynomial in \eqref{eq:summary-Sch-examples} together with \Cref{thm:add-hypers} \ref{item:thm_add-hypers:increase-deg}.
    Let $N,d$ be integers satisfying $$
        d \geq n' + m \quad \text{and} \quad 4 \leq N \leq n' + 2^{n'} - 2 + \sum_{l=0}^{n'-1}\binom{n'}{l}\left\lfloor\frac{l}{m}\right\rfloor.
    $$
    Then there exists an integer $2 \leq n \leq n'$ such that $$
        n + 1 \leq N - 1 < N \leq n + 2^{n} - 2 + \sum_{l=0}^{n-1}\binom{n}{l}\left\lfloor\frac{l}{m}\right\rfloor.
    $$
    Note that $d \geq n' + m \geq n + m$. If $d = n + m$, then there exists by \eqref{eq:summary-Sch-examples} an irreducible polynomial $f \in k[x_1,\dots,x_N,z]$ of degree $d$ satisfying condition \eqref{eq:condition-star2} and \eqref{eq:base-example-tor-condition}.
    So, we can assume $d > n + m$, then there exists by \eqref{eq:summary-Sch-examples} an irreducible polynomial $\tilde{f} \in k[x_1,\dots,x_{N-1},z]$ of degree $n + m$ satisfying condition \eqref{eq:condition-star2} and such that $$
        \Tor^{\Z/m}\left(\Spec k[x_1,\dots,x_{N-1},z]/(\tilde{f}),\Spec k[x_1,\dots,x_{N-1},z]/(\tilde{f},\partial_z \tilde{f})\right) = m.
    $$
    Note that this uses \Cref{lem:properties-of-torsion-order} \labelcref{item:lem_properties-tor-order:bigger-subset,item:lem_properties-tor-order:field-extension}. Applying \Cref{thm:add-hypers} \ref{item:thm_add-hypers:increase-deg} to the polynomial $\tilde{f}$ with $r = 0$ and added degree $d-n-m$ and $1$ added variable, there exists an irreducible polynomial $f \in k[x_1,\dots,x_N,z]$, denoted by $\check{f}$ in \Cref{thm:add-hypers}, of degree $d = n + m + (d - n - m)$ satisfying condition \eqref{eq:condition-star2} and such that $$
        \Tor^{\Z/m}\left(\Spec k[x_1,\dots,x_{N},z]/(f),\Spec k[x_1,\dots,x_{N},z]/(f,\partial_z f)\right) = m,
    $$
    i.e.~the polynomial $f$ satisfies the claims in the theorem.

    It remains to show the theorem for $N = 3$. Since $k$ is an uncountable field, there exists $\pi,\rho \in k$ which are algebraically independent over the prime field of $k$. We claim that for any $d \geq 2 + m$ the polynomial
    $$
        f := \rho x_3^d + \pi \left(1 + x_1^{\lceil \frac{3}{m} \rceil} + x_2^{\lceil \frac{3}{m} \rceil}\right)^m - x_1x_2 + x_1 x_3^m + x_1 x_2 z^m \in k[x_1,x_2,x_3,z]
    $$
    satisfies the properties stated in the theorem. Indeed, a direct computation shows that if $f \in k[x_1,x_2,x_3][z]$ would be reducible, then it has to be divisible by $x_1 x_2$, which it is clearly not. Thus $f$ is irreducible. 
    We note that $f$ degenerates via $\rho \to 0$ to a polynomial $f_0$ of the form \eqref{eq:Sch-example-f0-affine}, thus it follows from \Cref{lem:torsion-order-degeneration} that $$
        \Tor^{\Z/m}\left(\Spec k[x_1,x_2,x_3,z]/(f),\Spec k[x_1,x_2,x_3,z]/(f,\partial_z f)\right) = m,
    $$
    see also \cite[Corollary 6.2]{LS24}.
    The condition \eqref{eq:condition-star2} follows by a similar argument as in \Cref{lem:Sch-example-rationalpt}, which shows the claim and thus the theorem.
\end{proof}

\subsection{A special quartic fourfold}\label{sec:HPT-quartic}

We provide three examples of affine complete intersections, namely a quartic fourfold, a $(2,2,2)$-fourfold, and a $(3,3)$-fivefold, which also fit in our framework of affine degenerations from \Cref{sec:affine-degenerations}, which are used in the proof of \Cref{thm:intro-ci-2-torsion} for the case of small Fano index $r \leq 2$, see also \Cref{prop:aff-ci-low-index}.
These are based on a special quartic fourfold example from \cite[Example 4.6]{LS24}, which is birational to the quadric surface bundle example in \cite{HPT18}.
Throughout this section, let $k$ be an algebraically closed field of characteristic different from $2$.

\begin{example}\label{ex:HPT-quartic}
    Consider the irreducible polynomial $$
        h := y_1^2 + x_1x_2 y_2^2 + x_2 y_3^2 + x_1(1 + x_1^2 + x_2^2 - 2x_1 - 2x_2 - 2x_1x_2) \in k[x_1,x_2,y_1,y_2,y_3].
    $$
    The argument sketched in \cite[Example 4.6]{LS24} shows that the $\Z/2$-torsion order of the affine scheme $$
        X := \Spec k[x_1,x_2,y_1,y_2,y_3]/(h)
    $$
    relative to the closed subscheme $W := \{x_1x_2y_1 = 0\} \subset X$ is equal to $2$. In fact, the argument presented there can be adapted to show that \begin{equation*}
        \Tor^{\Z/2}\left(\Spec k[x_1,x_2,y_1,y_2,y_3]/(h),\Spec k[x_1,x_2,y_1,y_2,y_3]/(h,x_1x_2y_1y_2y_3)\right) = 2.
    \end{equation*}
    We would like to emphasize that this relies on some vanishing results for unramified cohomology classes, for instance \cite[Theorem 9.2]{Sch19JAMS}.
    
    By replacing $y_1 = x_1 z_1$, $y_2 = z_2$, and $y_3 = x_1 z_3$ and dividing the resulting polynomial by $x_1$, we find that $X$ is birational to the affine hypersurface associated to the irreducible polynomial
    \begin{equation}\label{eq:HPT-quartic}
        f := x_1 z_1^2 + x_2 z_2^2 + x_1 x_2 z_3^2 + (1 + x_1^2 + x_2^2 - 2x_1 - 2x_2 - 2x_1x_2) \in k[x_1,x_2,z_1,z_2,z_3].
    \end{equation}
    In particular, we see from the above explicit change of coordinates that
    \begin{equation}\label{eq:torsion-order-HPT-quartic}
        \Tor^{\Z/2}\left(\Spec k[x_1,x_2,z_1,z_2,z_3]/(f),\Spec k[x_1,x_2,z_1,z_2,z_3]/(f,x_1x_2z_1z_2z_3)\right) = 2.
    \end{equation}
\end{example}

To see that the polynomial $f$ fits in our framework of affine degeneration, we check the condition \eqref{eq:condition-star2} by a similar argument as in the proof of \Cref{lem:Sch-example-rationalpt}.

\begin{lemma}\label{lem:HPT-quartic-rational-pt}
    For any field extension $F/k$ and any $q \in F$, there exists an $F$-algebra epimorphism
    $$
        F[x_1,x_2,z_1,z_2,z_3,w]/(f+q,w\partial_{z_3}f-1) \longtwoheadrightarrow F,
    $$
    where $f \in k[x_1,x_2,z_1,z_2,z_3]$ is the polynomial in \eqref{eq:HPT-quartic}.
\end{lemma}

\begin{proof}
    Let $F/k$ be a field extension and let $q \in F$. If $q$ is defined over $k$, then there is an $k$-algebra epimorphism
    $$
        k[x_1,x_2,z_1,z_2,z_3,w]/(f+q,2 w x_1x_2z_3 - 1) \longtwoheadrightarrow k[z_1]/(z_1^2 + z_2^2 -2 + q) \longtwoheadrightarrow k,
    $$
    where the first morphism is the quotient by the ideal $(x_1 - 1,x_2 -1,z_3-1)$ and the second morphism exists as $k$ is algebraically closed. Thus the lemma follows for $q \in k \subset F$ by the right-exactness of the tensor product.
    Thus, we may assume $q \in F \setminus k$. 
    Then, the quotient by the ideal $(x_1 + x_2 -1,z_1-1,z_2,z_3-2)$ yields the claimed $F$-algebra epimorphism
    $$
        F[x_1,x_2,z_1,z_2,z_3,w]/(f+q,2wx_1x_2z_3-1) \longtwoheadrightarrow F[x_1,w]/(x_1 + q,4wx_1(x_1-1) + 1) \cong F,
    $$
    where we used that $q \neq 0,1$.
\end{proof}

We recall the explicit birational equivalence of the (affine) quartic fourfold in \Cref{ex:HPT-quartic} with a $(2,2,2)$-complete intersection from \cite[\S 3.1]{HPT18-3quadrics}, which is based on a construction of Beauville \cite{bea77}.

\begin{example}\label{ex:HPT-3quadrics-birational}
    Consider the polynomials \begin{equation}\label{eq:HPT-3quadrics}
        p_1 := -y_1 + z_2^2 + z_3 z_4 - 2, \quad p_2 := y_1 + y_2z_3 + z_1^2  - 2, \quad
            p_3 := y_1 z_5 - y_2 z_4 + 1 + z_5^2
    \end{equation}
    in $k[y_1,y_2,z_1,z_2,z_3,z_4,z_5]$, see \cite[Equation (3.3)]{HPT18-3quadrics}. 
    Then there exists an isomorphism of $k$-algebras
    $$
        k[y_1,y_2,z_1,z_2,z_3,z_4,z_5]_{z_3}/(p_1,p_2,p_3) \cong k[x_1,x_2,z_1,z_2,z_3]_{z_3}/(f), 
    $$
    where $f \in k[x_1,x_2,z_1,z_2,z_3]$ is as in \eqref{eq:HPT-quartic}. In particular, the affine $(2,2,2)$-complete intersection $\{p_1 = p_2 = p_3 = 0\} \subset \aff^7_k$ is birational to the quartic fourfold $\{f = 0\} \subset \aff^5_k$.
    Indeed, the following sequence of isomorphism provides the claimed isomorphism:
    $$
        \begin{aligned}
            &\hphantom{=}\ k[y_1,y_2,z_1,z_2,z_3,z_4,z_5]_{z_3}/( -y_1 + z_2^2 + z_3 z_4 - 2,y_1 + y_2z_3 + z_1^2  - 2,y_1 z_5 - y_2 z_4 + 1 + z_5^2) \\  
            &\cong k[z_1,z_2,z_3,z_4,z_5]_{z_3}/\left((z_2^2 + z_3 z_4 - 2) z_5 + (z_2^2 + z_3 z_4 + z_1^2 -4)z_3^{-1} z_4 + 1 + z_5^2\right) \\
            &= k[z_1,z_2,z_3,z_4,z_5]_{z_3}/\left(z_3^{-1} z_4 z_1^2 + (z_5 + z_3^{-1}z_4)z_2^2 + z_3 z_4 z_5 + z_4^2 + 1 + z_5^2 - 2 z_5 - 4 z_3^{-1} z_4\right) \\
            &\cong k[x_1,x_2,z_1,\dots,z_5]_{z_3}/\left(x_1 - z_3^{-1}z_4, x_2 - x_1 - z_5,x_1 z_1^2 + x_2z_2^2 + x_2z_3z_4 + (z_5-1)^2 - 4 x_1 \right) \\
            &\cong k[x_1,x_2,z_1,z_2,z_3]_{z_3}/\left(x_1 z_1^2 + x_2z_2^2 + x_1x_2z_3^2 + (x_2 -x_1 -1)^2 - 4 x_1 \right).
        \end{aligned}
    $$
\end{example}
    We could not verify the condition \ref{item:thm_add-hypers:strongly-rat} for any two of the three $p_i$'s. It seems that the polynomials $p_1,p_2,p_3$ in \eqref{eq:HPT-3quadrics} do not fit directly in the setup of \Cref{thm:add-hypers}.
    Instead, we consider a different affine chart of the projective closure of $\{p_1 = p_2 = p_3= 0 \} \subset \aff^7$.

    \begin{example}\label{ex:HPT-3quadrics-correct-chart}
        Consider the $k$-algebra $k[y_1,y_2,z_1,z_2,z_3,z_4,z_5]$ and the polynomials $p_1,p_2,p_3$ from \eqref{eq:HPT-3quadrics}. We perform the following change of coordinates
        $$
            (x_1,\dots,x_5,x_6,x_7) := (y_2 z_4^{-1},z_1 z_4^{-1},z_2 z_4^{-1}, z_3 z_4^{-1}, z_4^{-1},y_1 z_4^{-1},z_5 z_4^{-1}).
        $$
        Then the quadric polynomials $p_1,p_2,p_3$ from \eqref{eq:HPT-3quadrics} read
        \begin{equation}\label{eq:HPT-3quadrics-correct-chart}
        \begin{aligned}
            q_1 &:= -x_6x_5 + x_3^2 + x_4 - 2x_5^2 \in k[x_1,\dots,x_7], \\
            q_2 &:= x_6x_5 + x_1x_4 + x_2^2 - 2x_5^2 \in k[x_1,\dots,x_7], \\
            q_3 &:= x_6 x_7 - x_1 + x_5^2 + x_7^2 \in k[x_1,\dots,x_7].
        \end{aligned}
        \end{equation}
    \end{example}

    The above constructions imply that the affine $(2,2,2)$-complete intersection given by $q_1,q_2,q_3$ satisfies the assumptions in \Cref{thm:add-hypers} for $m = 2$.

\begin{corollary}\label{cor:HPT-3quadrics}
    Let $q_1,q_2,q_3 \in k[x_1,\dots,x_7]$ be as in \eqref{eq:HPT-3quadrics-correct-chart}. Then the following holds:
    \begin{enumerate}[label=(\roman*)]
        \item The $k$-algebra $ k[x_1,x_3,x_4,x_5,x_6,x_7]/(q_1,q_3) \cong k[x_3,x_5,x_6,x_7]$ is strongly $k$-rational;
        \item $X := \Spec k[x_1,\dots,x_7]/(q_1,q_2,q_3)$ is an integral $k$-variety of dimension $4$.
        \item Set $W := \{x_2 x_4 x_5 = 0\} \subset X$, then $\Tor^{\Z/2}(X,W) = 2$;
        \item For any field extension $F/k$ and $q \in F$, there exists an $F$-algebra epimorphism
        $$
            F[x_1,\dots,x_7,w]/(q_1,q_2 + q,q_3,w \partial_{x_2} q_2 - 1) \longtwoheadrightarrow F. 
        $$
    \end{enumerate}
\end{corollary}

\begin{proof}
    Statement (i) is obvious and statement (ii) follows from \cite{HPT18-3quadrics}. 
    The isomorphisms in \Cref{ex:HPT-3quadrics-birational} and the change of coordinates in \Cref{ex:HPT-3quadrics-correct-chart} imply that
    $$
        \begin{aligned}
        k[x_1,x_2,z_1,z_2,z_3]_{x_1z_3}/(f) & \longrightarrow k[x_1,\dots,x_7]_{x_4x_5}/(q_1,q_2,q_3), \\
        (x_1,x_2,z_1,z_2,z_3) &\longmapsto (x_4^{-1},x_7x_5^{-1} + x_4^{-1},x_2x_5^{-1},x_3x_5^{-1},x_4x_5^{-1})
        \end{aligned}
    $$
    is an isomorphism of $k$-algebras, where $f$ is as in \eqref{eq:HPT-quartic}. Via this isomorphism statement (iii) follows from \eqref{eq:torsion-order-HPT-quartic}.
    We turn to the proof of statement (iv). Let $F/k$ be a field extension and let $q \in F$. Note first $F[x_1\dots,x_7]/(q_1,q_2 + q,q_3)$ is equal to
    $$ 
            F[x_2,x_3,x_5,x_6,x_7]/(x_6x_5 + (x_5^2 + x_7^2 + x_6 x_7)(2x_5^2 + x_6 x_5 - x_3^2) + x_2^2 - 2x_5^2 +q).
    $$
    Thus the quotient by the ideal $(x_2-1,x_3 -1,x_5,x_7-1)$ induces the $F$-algebra epimorphism
    $$
        F[x_1,\dots,x_7,w]/(q_1,q_2 + q,q_3,2x_2 w - 1) \longtwoheadrightarrow F[x_6]/(-x_6 +  q) \cong F,
    $$
    which shows statement (iv) and thus completes the proof of the corollary.
\end{proof}

We conclude this section by constructing an affine $(3,3)$-fivefold whose relative $\Z/2$-torsion order is divisible by $2$. By applying the results from \Cref{sec:affine-degenerations} to an affine $(2,3)$-complete intersection isomorphic to the affine quartic fourfold in \Cref{ex:HPT-quartic}.

\begin{example}\label{ex:33-5fold}
From the explicit form of the polynomial $f$ in \eqref{eq:HPT-quartic}, we see that there exists an $k$-algebra isomorphism
$$
    k[x_1,x_2,z_1,z_2,z_3]/(f) \cong k[x_1,x_2,x_3,z_1,z_2,z_3]/(f_1,f_2)
$$
where $f_1,f_2 \in k[x_1,x_2,x_3,z_1,z_2,z_3]$ are the polynomials \begin{equation}\label{eq:23-ci}
        f_1 := x_3 - z_3^2, \quad f_2 := x_1 z_1^2 + x_2 z_2^2 + x_1x_2 x_3 + (1 + x_1^2 + x_2^2 - 2x_1 - 2x_2 - 2x_1x_2).
\end{equation}

Note that the $k$-algebra $B := k[x_1,x_2,x_3,z_1,z_2]_{x_1x_2}/(f_2)$ is strongly $k$-rational and that $B[z_3]/(f_1)$ is a geometrically integral $k$-algebra of Krull dimension $4$, as it is a localization of a geometrically integral $k$-algebra. Moreover, we have for any field extension $F/k$ and any element $q \in F$ an $F$-algebra epimorphism
\begin{equation}\label{eq:23-ci-rational-pt}
    B \otimes_k F[z_3]_{\partial_{z_3} f_1}/(f_1 + q) = k[x_1,x_2,x_3,z_1,z_2,z_3]_{x_1x_2z_3}/(f_1 + q ,f_2) \twoheadrightarrow F
\end{equation}
given by the ideal $(z_3 - \alpha,x_2 - x_3, x_1 - x_2 -1, z_1 - \sqrt{-1} x_3, z_2 - 2)$, where $\alpha \in k^\ast$ is chosen such that $\alpha^2 + q \neq -1,0$.
In particular, we find that $f_1 + x_4 \in B[z_3,x_4]$ and $x_4 \in B[x_4]$ are admissible with respect to $B[x_4]$ and satisfy condition \eqref{eq:condition-star} by \Cref{ex:adding-hyperplane2}.

Thus, we can apply \Cref{thm:induction-step} and find by \eqref{eq:torsion-order-HPT-quartic} that the affine scheme 
$$
    X' := \Spec k(t)[x_1,x_2,x_3,x_4,z_1,z_2,z_3]_{x_1x_2}/(t-x_4(f_1 + x_4),f_2),
$$
is a geometrically integral $k(t)$-variety of dimension $5$ satisfying $$
    \Tor^{\Z/2}(X' \times_{k(t)} \overline{k(t)},W' \times_{k(t)} \overline{k(t)}) = 2,
$$
where $W' := \{x_1x_2z_1z_3 = 0\} \subset X'$ and $k(t)/k$ is a purely transcendental field extension of transcendence degree $1$.
In fact, the $k(t)$-scheme
$$
    X := \Spec k(t)[x_1,x_2,x_3,x_4,z_1,z_2,z_3]/(t-x_4(f_1 + x_4),f_2)
$$
is an integral $k$-variety by a similar argument as in \Cref{prop:general-deg-strictly-semi-stable} and hence of dimension $5$. Moreover, $X$ is Cohen-Macaulay by \cite[Proposition 18.13]{Eis95} and contains an open dense geometrically integral subscheme. Thus $X$ is geometrically integral and it satisfies $$
    \Tor^{\Z/2}(X \times_{k(t)} \overline{k(t)},W \times_{k(t)} \overline{k(t)}) = 2
$$
for the closed subscheme $W := \{x_1x_2z_1z_3 = 0\} \subset X$.
\end{example}

\begin{corollary}\label{cor:HPT-2cubics}
    Let $f_1,f_2 \in k[x_1,x_2,x_3,z_1,z_2,z_3]$ be the polynomials as in \eqref{eq:23-ci} and let $K := \overline{k(t)}$ be the algebraic closure of a purely transcendental field extension $k(t)/k$ of transcendence degree $1$. Then the following holds:
    \begin{enumerate}[label=(\roman*)]
        \item The $K$-algebra $B:=K[x_1,x_2,x_3,x_4,z_2,z_3]/(t - x_4(f_1+x_4))$ is strongly $K$-rational;
        \item $X := \Spec B[z_1]/(f_2)$ is an integral $K$-variety of dimension $5$;
        \item Set $W := \{x_1x_2z_1z_3 = 0\} \subset X$, then $\Tor^{\Z/2}(X,W) = 2$;
        \item For any field extension $F/k$ and $q \in F$, there exists an $F$-algebra epimorphism
        $$
            F[x_1,x_2,x_3,x_4,z_1,z_2,z_3]_{\partial_{z_1} f_2}/(t - x_4(f_1 + x_4),f_2 + q) \longtwoheadrightarrow F.
        $$
    \end{enumerate}
\end{corollary}

\begin{proof}
    The first statement follows directly from the isomorphism of $K$-algebras
    $$
        B = K[x_1,x_2,x_3,x_4,z_2,z_3]/(t - x_4(x_3 - z_3^2+x_4)) \cong K[x_1,x_2,x_4,z_2,z_3]_{x_4}.
    $$
    The statement (ii) and (iii) are proven in \Cref{ex:33-5fold}. To prove statement (iv), let $F/K$ be a field extension and let $q \in F$. Then the quotient by the ideal $$
        \left(x_4-1,z_2-2,1-x_1+x_2,z_1-1,z_3 - \sqrt{1-t}\right)
    $$
    yields for $q \neq 0$ an $F$-algebra epimorphism
    $$
        F[x_1,x_2,x_3,x_4,z_1,z_2,z_3]_{\partial_{z_1} f_2}/(t - x_4(f_1 + x_4),f_2 + q) \longtwoheadrightarrow F[x_1]/(x_1 + q) \cong F.
    $$
    For the $q = 0$, the existence of such an $F$-algebra epimorphism follows from \eqref{eq:23-ci-rational-pt}.
\end{proof}

\section{Applications and main results}\label{sec:applications}

In this section we apply the machinery developed in \Cref{sec:affine-degenerations} to the base examples in \Cref{sec:base-examples} in order to construct certain affine complete intersections. Taking the closure of these complete intersections in different compactification of affine space $\aff^N$ enables us to prove the main results stated in the introduction.

\subsection{Complete intersections}

We start with the most straightforward compactification of $\aff^N$, namely $\CP^N$. 
In the following proposition, we construct affine complete intersections with given (relative) torsion order.
Their projective closure can be controlled via the theory of Gröbner basis and graded monomial orderings, see \Cref{prop:homogenization-monomial-order}.

\begin{proposition}\label{prop:aff-ci}
    Let $k$ be an uncountable, algebraically closed field and let $n,m \geq 2$ be two integers such that $m \in k^\ast$. For $s \geq 1$, consider an $s$-tuple $(d_1,\dots,d_s) \in \Z_{\geq 1}^s$ of positive integers satisfying $d_1 \geq n + m$.
    
    Then for all integers $N$ and $M$ satisfying the inequalities
    \begin{equation}\label{eq:prop_aff-ci:dimension-bound-hypersurfaces}
        s-1 \leq M \leq \sum\limits_{i=1}^s d_i - n - m \quad \text{and} \quad
        4 \leq N \leq n + 2^n - 1 + \sum\limits_{j=0}^{n-1} \binom{n}{j} \left\lfloor \frac{j}{m} \right\rfloor,  
    \end{equation}
    there exists polynomials $g_1,\dots,g_s \in k[x_1,\dots,x_N,y_1,\dots,y_{M}]$ of degree $\deg g_i = d_i$ such that the scheme $$
            X := \Spec k[x_1,\dots,x_N,y_1,\dots,y_{M}]/(g_1,\dots,g_s)
        $$
    is an (integral) $k$-variety of dimension $N + M - s$ and satisfies $\Tor^{\Z/m}(X,W) = m$,
    where $W := \{l \cdot \partial_{x_1} g_1 = 0\} \subset X$ for some non-zero polynomial $l \in k[x_1,\dots,x_{N}]$. 
    
    Moreover, if $s \leq M$ then the leading monomials of $g_1,\dots,g_s$ are relatively prime for some graded monomial ordering.
\end{proposition}

\begin{proof}
    Let $s \geq 1$ and let $(d_1,\dots,d_s) \in \Z_{\geq 1}^s$ be a tuple of positive integers such that $d_1 \geq n + m$.
    We can assume without loss of generality that each $d_i \geq 2$, as $d_1 \geq n+m = 4$ and for any $k$-algebra $A$ we have a $k$-algebra isomorphism $A[z]/(z) \cong A$.
    Let $N,M,d$ be positive integers satisfying the inequalities in \eqref{eq:prop_aff-ci:dimension-bound-hypersurfaces}.

    Since $N$ satisfies the bounds in \eqref{eq:prop_aff-ci:dimension-bound-hypersurfaces}, \Cref{thm:base-examples-hypersurfaces} shows that for every integer $d' \geq n+m$ there exist non-zero polynomials
    \begin{equation}\label{eq:prop_aff-ci:existence-hypers}
        f \in k[x_1,\dots,x_N] \quad \text{and} \quad l \in k[x_1,\dots,x_N]
    \end{equation}
    such that $f$ is irreducible of degree $d'$ and satisfies condition \eqref{eq:condition-star2} with $z = x_1$ as well as
    \begin{equation}\label{eq:prop_aff-ci:tor-order-hypersurface}
        \Tor^{\Z/m}\left(\Spec k[x_1,\dots,x_N]/(f),\Spec k[x_1,\dots,x_N]/(f,l \cdot \partial_{x_1} f)\right) = m.
    \end{equation}
    In particular, the properties \labelcref{item:thm_add-hypers:integral,item:thm_add-hypers:ratl-pt,item:thm_add-hypers:strongly-rat,item:thm_add-hypers:tor-order} in \Cref{thm:add-hypers} are satisfied for $f$ and the integers $n = N-1$ and $r = 0$.
    We will use the existence of such $f$ several times throughout the proof.

    If $s = 1$, then we set $d' := d_1 - M$ and note that $d' \geq m+n$ by \eqref{eq:prop_aff-ci:dimension-bound-hypersurfaces}. Consider the irreducible polynomial $f \in k[x_1,\dots,x_N]$ of degree $d'$ from \eqref{eq:prop_aff-ci:existence-hypers}. 
    Then, we define a polynomial $g_1 \in k[x_1,\dots,x_N,w_1,\dots,w_M]$ of degree $d_1$ by setting $g_1 := f$, if $M = 0$, or by applying \Cref{thm:add-hypers} \ref{item:thm_add-hypers:increase-deg} to $f$ and with $r = 0$, added degree $M$ and $M$ added variables and setting $g_1 := \check{f}$, if $M > 0$. In both cases we know that the affine scheme
    $$
        \Spec k[x_1,\dots,x_N,y_1,\dots,y_M]/(g_1)
    $$
    is an integral $k$-variety of dimension $N + M - 1$ and satisfies $\Tor^{\Z/m}(X,W) = m$,
    where $W := \{l \cdot \partial_{x_1} g_1 = 0\} \subset X$ for $l$ as in \eqref{eq:prop_aff-ci:existence-hypers}. This proves the proposition for $s = 1$. 
    
    Assume from now on $s \geq 2$. We prove the proposition by distinguishing three cases depending on $M$.

    \textbf{Case a.} If $M \geq 2s - 2$, we set $d' := \min \{d_1,d_1+ \dots + d_s - M\}$. 
    Note that $n + m \leq d' \leq d_1$ by \eqref{eq:prop_aff-ci:dimension-bound-hypersurfaces}.
    Since each $d_i \geq 2$ by assumption, we have 
    \begin{equation*}
        2s - 2 \leq M + d' - d_1 \leq d_2 + \dots +d_s.
    \end{equation*}
    Thus there exist (not necessarily unique) integers $M_2,\dots,M_s \in \Z_{\geq 2}$ such that $2 \leq M_i \leq d_i$ and $M_2 + \dots + M_r = M + d' - d_1$.
    We set $M_1 := d_1 - d'$. Let $f \in k[x_1,\dots,x_N]$ be the irreducible polynomial of degree $d'$ as in \eqref{eq:prop_aff-ci:existence-hypers}.
    By applying \Cref{thm:add-hypers} \ref{item:thm_add-hypers:increase-deg} to $f$ with added degree $M_1$ and $M_1$ added variables
    if $M_1 > 0$, we can assume that there exists an irreducible polynomial $\check{f} \in k[x_1,\dots,x_{N},y_1,\dots,y_{M_1}]$
    of degree $d_1$, which satisfies the properties \labelcref{item:thm_add-hypers:integral,item:thm_add-hypers:ratl-pt,item:thm_add-hypers:strongly-rat,item:thm_add-hypers:tor-order} in \Cref{thm:add-hypers}.
    Then we repeatedly apply \Cref{thm:add-hypers} \ref{item:thm_add-hypers:add-hypers} to $\check{f}$ with added degrees $d_2,\dots,d_s$ and $M_2,\dots,M_s$ added variables in order to obtain polynomials
    \begin{equation}\label{eq:prop_aff-ci:Case-a-gi}
        \begin{aligned}
        &g_1 \in k[x_1,\dots,x_N,y_1,\dots,y_M], \, g_2 \in k[y_{M_1 + 1},\dots,y_{M_1+M_2}], \\
        &g_3 \in k[y_{M_1+M_2+1},\, \dots,y_{M_1+M_2+M_3}],\dots,\, g_s \in k[y_{M-M_s+1},\dots,y_M]
        \end{aligned}
    \end{equation}
    of degree $\deg g_i = d_i$ such that the scheme
    $$
        X := \Spec k[x_1,\dots,x_{N},y_1,\dots,y_M]/(g_1,\dots,g_s)
    $$
    is an integral $k$-variety of dimension $N + M - s$ (see \ref{item:thm_add-hypers:integral}) and satisfies $\Tor^{\Z/m}(X,W) = m$, where $W := \{l \cdot \partial_{x_1} g_1 = 0\} \subset X$ for some $l \in k[x_1,\dots,x_{N}]$ (see \ref{item:thm_add-hypers:tor-order}). In fact, $l$  can be choosen as in \eqref{eq:prop_aff-ci:existence-hypers}, see \Cref{rem:to-add-hypers}.
    The claim about the leading monomials follows directly from \eqref{eq:prop_aff-ci:Case-a-gi} by choosing the graded lexicographical ordering, see \Cref{example:grlex-ordering}.
    Note that this uses implicitely that the degree of $g_1$ does not change under the construction in \Cref{thm:add-hypers} \ref{item:thm_add-hypers:add-hypers}, see \Cref{rem:to-add-hypers}.

    \textbf{Case b.} Assume $s \leq M < 2s - 2$. Define the integer $s' := M + 2 - s$ and note that $2 \leq s' < s$ by the assumption on $M$.
    Similar to Case a, we apply \Cref{thm:add-hypers} \ref{item:thm_add-hypers:add-hypers} repeatedly $s'-1$ times with added degrees $d_2,\dots,d_{s'}$ and $2,2,\dots,2$ added variables starting with the polynomial $f \in k[x_1,\dots,x_N]$ of degree $d' := d_1$ as in \eqref{eq:prop_aff-ci:existence-hypers}.
    Thus, we obtain polynomials
    \begin{equation}\label{eq:prop_aff-ci:Case-b-somegi}
        g_1 \in k[x_1,\dots,x_N,y_1,\dots,y_{2s'-2}], \, g_2 \in k[y_1,y_2], \, \dots, \, g_{s'} \in k[y_{2s'-3},y_{2s'-2}]
    \end{equation}
    of degree $\deg g_i = d_i$ as well as $\deg_x g_1 = d_1$ such that
    $$
        X' := \Spec k[x_1,\dots,x_N,y_1,\dots,y_{2s'-2}]/(g_1,\dots,g_{s'})
    $$
    is an integral $k$-variety of dimension $N + s' - 2 = N + M - s$ (see \ref{item:thm_add-hypers:integral}) and satisfies $\Tor^{\Z/m}(X',W') = m$, where $W' := \{l \cdot \partial_{x_1} g_1 = 0\} \subset X'$ for $l' \in k[x_1,\dots,x_N]$ as in \eqref{eq:prop_aff-ci:existence-hypers} (see \ref{item:thm_add-hypers:tor-order}).
    By \Cref{rem:to-add-hypers} \ref{item:rem:add-hypers_alternative-for-M2}, we can assume that
    \begin{equation}\label{eq:prop_aff-ci:Case-b-explicit-form-gs'}
        g_{s'} = a - y_{2s'-3}y_{2s'-2} + b y_{2s'-3}^{d_{s'}} \in k[y_{2s'-3},y_{2s'-2}].
    \end{equation}
    for some elements $a,b \in k^\ast$. We define additionally the polynomials
    \begin{equation}\label{eq:prop_aff-ci:Case-b-additionalgi}
        g_{s' + 1} := y_{2s'-1} + y_{2s'-2}^{d_{s'+1}}, \, g_{s'+2} := y_{2s'}  + y_{2s'-1}^{d_{s'+2}}, \, \dots, \, g_s := y_{s + s' - 2} + y_{s + s' - 3}^{d_s}
    \end{equation}
    in $k[y_{2s-1},\dots,y_M]$ of degree $\deg g_i = d_i$. Then there is an obvious isomorphism of affine $k$-schemes
    $$
        \begin{aligned}
        X &:= \Spec k[x_1,\dots,x_N,y_1,\dots,y_M]/(g_1,\dots,g_s) \\ &\cong \Spec k[x_1,\dots,x_N,y_1,\dots,y_{2s'-2}]/(g_1,\dots,g_{s'}) = X',
        \end{aligned}
    $$
    where $g_1,\dots,g_{s'}$ are as in \eqref{eq:prop_aff-ci:Case-b-somegi}. In particular, $X$ is an affine $k$-variety of dimension $N + M - s$.
    We note that the isomorphism sends $W' \subset X'$ to $W := \{l \cdot \partial_{x_1} g_1 = 0\} \subset X$, which immediately implies that $\Tor^{\Z/m}(X,W) = \Tor^{\Z/m}(X',W') = m$.
    The relative primeness of the leading monomial of $g_1,\dots,g_s$ for the graded lexicographical ordering in \Cref{example:grlex-ordering} follows from the explicit forms of $g_1,\dots,g_s$ in \eqref{eq:prop_aff-ci:Case-b-somegi}, \eqref{eq:prop_aff-ci:Case-b-explicit-form-gs'}, and \eqref{eq:prop_aff-ci:Case-b-additionalgi}.

    \textbf{Case c.} If $M = s-1$, then consider the polynomial $g_1 := f \in k[x_1,\dots,x_N]$ as in \eqref{eq:prop_aff-ci:existence-hypers} and define the polynomials
    $$
        g_2 := y_1 + h_2(x_1,\dots,x_N), \, g_3 := y_2 + h_3(x_1,\dots,x_N), \, \dots, \, g_s := y_M + h_s(x_1,\dots,x_N)
    $$
    in $k[x_1,\dots,x_N,y_1,\dots,y_M]$, where $h_2,\dots,h_s \in k[x_1,\dots,x_N]$ are arbitrary polynomials of degree $\deg h_i = d_i$. Then the claims in the proposition reduce in this case via the obvious isomorphism of affine $k$-schemes
    $$
        \Spec k[x_1,\dots,x_N,y_1,\dots,y_{M}]/(g_1,\dots,g_s) \cong \Spec  k[x_1,\dots,x_N]/(f) =: X'
    $$
    to the same statements about $X'$, which are known by \Cref{thm:base-examples-hypersurfaces}.
\end{proof}

As a first consequence, we obtain a new lower bound on the torsion order of complete intersections in projective space. Recall that any (smooth) complete intersection of dimension at least $2$ in projective space satisfies $h^{1,0} = 0$, see e.g.\ \cite[Expos\'e XI, Th\'eor\`eme 1.5]{SGA7}. In particular, its torsion order is infinite or it has trivial Chow group of zero-cycles ($\CH_0(-) \cong \Z$), see \Cref{lem:torsion-order-and-CH0}.

\begin{theorem}\label{thm:ci-torsion-order}
    Let $k$ be a field and let $n,m \geq 2$ be integers. Then the torsion order of a very general complete intersection of multidegree $(d_1,\dots,d_s)$ and dimension $D \geq 4$ is divisible by $m$,
    if $d_1 \geq n + m$  and the Fano index
    \begin{equation}\label{eq:ci-bound-Fano-index}
        r := D + s + 1 - \sum\limits_{i=1}^s d_i \leq 2^n + \sum\limits_{j=0}^{n-1} \binom{n}{j} \left\lfloor \frac{j}{m} \right\rfloor - m.
    \end{equation}
\end{theorem}

\begin{proof}
    Since the torsion order of any $k$-variety is divisible by the torsion order of its base change to any field extension $L/k$ and the notion of very general is stable under field extension by \Cref{lem:very-general} \ref{item:very-gen:field-ext}, 
    we can assume without loss of generality that the field $k$ is algebraically closed and uncountable.

    Let $(d_1,\dots,d_s) \in \Z_{\geq 1}^s$ and $D \geq 4$ be positive integers satisfying $d_1 \geq n + m$ and
    $$
         D + s + 1 - \sum\limits_{i=1}^s d_i \leq 2^n + \sum\limits_{j=0}^{n-1} \binom{n}{j} \left\lfloor \frac{j}{m} \right\rfloor - m.
    $$
    Then we easily find integers $N$ and $M$ with $$
        s \leq M \leq \sum\limits_{j=1}^s d_i - n - m \quad \text{and} \quad 4 \leq N \leq n + 2^n - 1 + \sum\limits_{j= 0}^{n-1} \binom{n}{j} \left\lfloor \frac{j}{m} \right\rfloor
    $$
    such that $D = N + M - s$. 
    By \Cref{prop:aff-ci} with $N$ and $M$ as above, we know that there exist polynomials $$
        f_1,\dots,f_s \in k[x_1,\dots,x_N,y_1,\dots,y_M]
    $$
    of degree $\deg f_i = d_i$ such that the leading monomials of $f_1,\dots,f_s$ are relatively prime for some graded monomial ordering and 
    the affine $k$-variety
    $$
        X^\circ := \Spec k[x_1,\dots,x_N,y_1,\dots,y_M]/(f_1,\dots,f_s)
    $$
    has dimension $D$ and satisfies $\Tor^{\Z/m}(X^\circ,W^\circ) = m$ for some closed subscheme $W^\circ \subset X^\circ$.

    Let $X$ be the projective closure of $X^\circ$ in $\CP^{N+M}$.
    Then it follows immediately that $X$ is a (projective) $k$-variety of dimension $D$ satisfying 
    \begin{equation}\label{eq:thm_ci:torsion-order-example}
        \Tor^{\Z/m}(X,W) = m,
    \end{equation}
    where $W \subset X$ is the unique closed and reduced subscheme satisfying $X \setminus W = X^\circ \setminus W^\circ$ in $X$.
    Moreover, \Cref{prop:homogenization-monomial-order} implies that
    $$
        X = \Proj k[x_0,\dots,x_N,y_1,\dots,y_M]/(f_1^h,\dots,f_s^h),
    $$
    where $f_i^h \in k[x_0,\dots,x_N,y_1,\dots,y_M]$ is the homogenization of the polynomial $f_i$, see \Cref{def:homogenization}. 
    Thus $X$ is an integral complete intersection of multi-degree $(d_1,\dots,d_s)$ in $\CP^{N+M} = \CP^{D + s}$ such that the $\Z/m$-torsion order relative to some non-empty closed subscheme is equal to $m$.
    
    Let $X_{d_1,\dots,d_s}$ be a very general complete intersection of multidegree $(d_1,\dots,d_s)$ in $\CP^{D+s}$ over $k$. 
    We aim to show that $m$ divides $\Tor(X_{d_1,\dots,d_s})$ by using \eqref{eq:thm_ci:torsion-order-example}.
    By \Cref{lem:very-general} \ref{item:very-gen:degeneration}, up to a base change to an algebraically closed field extension of $k$ the variety $X_{d_1,\dots,d_s}$ degenerates to $X$.
    Since the torsion order is stable under base changes to algebraically closed field extensions, see \Cref{lem:properties-of-torsion-order}, we can assume that $X_{d_1,\dots,d_s}$ degenerates to $X$.
    The intersection of $D$ general hyperplane sections through a closed point of $W \subset X$ yields a closed subscheme $\mathcal{W}$ of the total space of the degeneration, which has relative dimension $0$ and whose restriction to $X$ is contained in $W$.
    Thus, we conclude from \Cref{lem:torsion-order-degeneration} applied to the total space of the degeneration with closed subscheme $\mathcal{W}$ that $$
        \Tor^{\Z/m}(X_{d_1,\dots,d_s},W_{d_1,\dots,d_s}) = \Tor^{\Z/m}(X,W) \overset{\eqref{eq:thm_ci:torsion-order-example}}{=} m,
    $$
    where $W_{d_1,\dots,d_s} := \mathcal{W} \cap X_{d_1,\dots,d_s}$ is a non-empty zero-dimensional closed subscheme of $X_{d_1,\dots,d_s}$.
    In particular, we find that $m$ divides the torsion order $\Tor(X_{d_1,\dots,d_s})$ by \Cref{lem:properties-of-torsion-order}, as the torsion order of $X_{d_1,\dots,d_s}$ is infinite or $\CH_0(X_{d_1,\dots,d_s}) \cong \Z$.
\end{proof}

\begin{proof}[Proof of \Cref{thm:intro-torsion-order-ci}:]
    Let $r,m \in \Z_{\geq 1}$ be positive integers. If $m = 1$, then the statement of the theorem is trivial, so we can assume that $m \geq 2$. We aim to deduce the theorem from \Cref{thm:ci-torsion-order}.
    Up to reordering $(d_1,\dots,d_s) \in \Z_{\geq 1}^s$, we can assume that $d_1 \geq \log_2 (r + m) + m$.
    We set $n := d_1 - m$ and note that
    $$
        n \geq \lceil \log_2(r+m) \rceil \geq \lceil \log_2(3) \rceil = 2.
    $$
    Note also that the dimension of a complete intersection $X$ of multidegree $(d_1,\dots,d_s)$ with positive Fano index $r > 0$ is at least
    $$
        \dim X \geq \sum\limits_{i=1}^s d_i - s \geq d_1 - 1 \geq \lceil \log_2(1+m) \rceil + m - 1 \geq 3.
    $$
    We observe that $\dim X = 3$ if and only if $s = 1$, $m = 2$ and $d_1 = 4$. 
    Note that the statement for quartic threefolds has been proven in \cite{CTP16}, see also \cite[Theorem 1.1]{Sch21-torsion}.
    Thus we can assume that $\dim X \geq 4$. Since
    $$
        r \leq 2^{n} - m \leq 2^n + \sum\limits_{j=0}^{n-1} \binom{n}{j} \left\lfloor \frac{j}{m} \right\rfloor - m,
    $$
    \Cref{thm:intro-torsion-order-ci} follows directly from \Cref{thm:ci-torsion-order}.
\end{proof}

\subsection{Complete intersections of small Fano index}

Starting with the examples in \Cref{sec:HPT-quartic}, we obtain additional affine complete intersection of small Fano index.

\begin{proposition}\label{prop:aff-ci-low-index}
    Let $k$ be an uncountable algebraically closed field of characteristic different from $2$. 
    Let $(d_1,\dots,d_s) \in \Z_{\geq 2}^s$ be an $s$-tuple for some integer $s \geq 1$.
    Then for all integers $M$ satisfying $$s \leq M \leq d_1 + \dots +d_s - 3,$$ there exist polynomials
    $f_1,\dots,f_s \in k[x_1,\dots,x_{4+M}]$
    of degree $\deg f_i = d_i$ whose leading monomials are relative prime for some graded monomial ordering
    such that
    $$
        X := \Spec k[x_1,\dots,x_{4+M}]/(f_1,\dots,f_s)
    $$
    is a (integral) $k$-variety of dimension $4 + M - s$ and satisfies $\Tor^{\Z/2}(X,W) = 2$ for some non-empty closed subscheme $W \subset X$.
\end{proposition}

The construction of $f_1,\dots,f_s$ is very similar to the construction in \Cref{prop:aff-ci} replacing the base example from \Cref{sec:hypersurface-examples} with the examples from \Cref{sec:HPT-quartic}. We provide some details for the convenience of the reader.

\begin{proof}
    Up to reordering the $d_i$'s, we can assume that $d_1 \geq d_2 \geq \dots \geq d_s$. If $d_1 \geq 4$, then the proposition is a special case of \Cref{prop:aff-ci} with $n = m = 2$ and $N = 4$, except if $M = d_1 + \dots + d_s - 3$, then take $N = 5$. So we can assume from now on that each $d_i \in \{2,3\}$. We distinguish three cases.

    \textbf{Case a.} Assume $d_1 = \dots = d_s = 2$ and $s \geq 3$. 
    Define $s_0 := 2s - M$ and note that $3 \leq s_0 \leq s$. Let $q_1,q_2,q_3 \in k[x_1,\dots,x_7]$ be the polynomials as in \eqref{eq:HPT-3quadrics-correct-chart} and define the polynomial
    $$
        q_i := x_{4+i} + x_{3+i}^2 \in k[x_1,\dots,x_{4+s_0}].
    $$
    for $i \in \{4,\dots,s_0\}$.
    Then we have an isomorphism of $k$-algebras $$
        k[x_1,x_3,\dots,x_{4+s_0}]/(q_1,q_3,q_4\dots,q_{s_0}) \cong k[x_1,x_3,\dots,x_7]/(q_1,q_3).
    $$
    Thus, it follows from \eqref{cor:HPT-3quadrics} that the assumption \labelcref{item:thm_add-hypers:integral,item:thm_add-hypers:strongly-rat,item:thm_add-hypers:tor-order,item:thm_add-hypers:ratl-pt} of \Cref{thm:add-hypers} are satisfied for $q_1,q_3,q_4,\dots,q_{s_0}$ and $f = q_2$ with $z = x_2$. Applying \Cref{thm:add-hypers} \ref{item:thm_add-hypers:add-hypers} with added degree $2$ and $2$ added variables $(s-s_0)$ times yields quadratic polynomials
    $$
    \begin{aligned}
        &f_1 \in k[x_0,\dots,x_{4 + 2s - s_0}], \quad f_2 := q_1 \in k[x_0,\dots,x_{4+s_0}], \\ 
        &f_i := q_i \in k[x_0,\dots,x_{s_0}], \quad f_{s_0 + j} \in k[x_{4 + s_0 + 2j -1},x_{4+s_0 + 2j}]
    \end{aligned}
    $$
    for $i \in \{3,\dots,s_0\}$ and $j \in \{1,\dots,s - s_0\}$ such that the affine scheme
    $$
        X := \Spec k[x_0,\dots,x_{4+M}]/(f_1,\dots,f_s)
    $$
    is an integral $k$-variety of dimension $4+M-s$ and $\Tor^{\Z/2}(X,W) = 2$ for some non-empty closed subscheme $W \subset X$.

    Note that the leading monomials of the polynomials $q_1,\dots,q_{s_0}$ with respect to the graded monomial ordering given by $x_2 > x_3 > \dots > x_{4+s_0} > x_1$ are
    $$
        \LM(q_1) = x_3^2, \quad \LM(q_2) = x_2^2, \quad \LM(q_3) = x_5^2, \quad \LM(q_i) = x_{3+i}^2
    $$
    for $i \in \{4,\dots,s_0\}$; in particular the leading monomials of $q_1,\dots,q_{s_0}$ are relatively prime. It follows from \Cref{thm:add-hypers} that the leading monomials of the polynomials $f_1,\dots,f_s$ are relatively prime with respect to the graded monomial ordering given by 
    $$
        x_2 > x_3 > \dots > x_{4+M} > x_1,
    $$
    see also \Cref{rem:to-add-hypers} \ref{item:rem:add-hypers_alternative-for-M2}. Thus, we can assume now $d_1 = 3$.

    \textbf{Case b.} Assume $d_1 = 3$ and $M \neq 3s -3$.
    We aim to apply a similar argument as in \Cref{prop:aff-ci}, which we spell out for the convenience of the reader.

    Let $f \in k[x_1,x_2,z_1,z_2,z_3]$ be the polynomial of degree $4$ as in \eqref{eq:HPT-quartic}.
    It is shown in \Cref{ex:HPT-quartic} and \Cref{lem:HPT-quartic-rational-pt} that the polynomial $f$ satisfies the assumption \labelcref{item:thm_add-hypers:integral,item:thm_add-hypers:strongly-rat,item:thm_add-hypers:tor-order,item:thm_add-hypers:ratl-pt} of \Cref{thm:add-hypers} for $r = 0$. 
    Define $$
            s_1 := \max \{e \in \{0,1,\dots,s-2\} : M - s - e \geq 0\}
    $$
    and set $M_0 := M - s + s_1 \in \Z$. Note that $2s_1 \leq M_0 \leq d_2 + d_3 + \dots + d_{s_1 + 1}$.
    Then there exists integers $M_2,\dots,M_{s_1+1} \in \Z_{\geq 2}$ such that
    $$
        M_2 + \dots + M_{s_1 +1} = M_0.
    $$
    By applying \Cref{thm:add-hypers} \ref{item:thm_add-hypers:add-hypers} $s_1$-times to the polynomial $f$ with added degree $d_2,\dots,d_{s_1+1}$ and $M_2,\dots,M_{s_1+1}$ added variable, we find polynomials
    \begin{equation}\label{eq:prop:small-index-first-poly}
        \begin{aligned}
        &\tilde{f} \in k[x_1,x_2,z_1,z_2,z_3,y_1,\dots,y_{M_0}], \ f_2 \in k[y_1,\dots,y_{M_2}], \\
        &f_3 \in k[y_{M_2 + 1},\dots,y_{M_2 + M_3}], \ \dots\ , \ f_{s_1 + 1} \in k[y_{M_0-M_{s_1+1}},\dots,y_{M_{s_1+1}}]
        \end{aligned}
    \end{equation}
    of degree $\deg \tilde{f} = 4$ and $\deg f_i = d_i$ for $i =2,\dots,s_1+1$ such that 
    $$
        X := \Spec  k[x_1,x_2,z_1,z_2,z_3,y_1,\dots,y_{M_0}]/(\tilde{f},f_2,f_3,\dots,f_{s_1+1})
    $$
    is a $k$-variety of dimension $M_0+ 4 - s_1$ satisfying $\Tor^{\Z/2}(X,W) = 2$ for some non-empty closed subscheme $W \subset X$. Moreover, the polynomial $\tilde{f}$ is of the form $\tilde{f} = f + h$ for some linear polynomial $h \in k[y_1,\dots,y_{M_0}]$, see \Cref{rem:to-add-hypers} \ref{item:rem:add-hypers_tildef}.

    For each $j \in \{s_1 + 2, \dots,s-1\}$ define the polynomial
    \begin{equation}\label{eq:prop:small-index-linear-poly}
        f_{j} := \begin{cases}
            z_2^{d_j} + z_4 & \text{if}\ j = s_1 + 2, \\
            z_{j+1-s_1}^{d_j} + z_{j+2-s_1} & \text{otherwise.}
        \end{cases}
    \end{equation}
    in $ k[z_2,z_4,z_5,\dots,z_{s+1-s_1}]$.
    Then we have
    $$
        \Spec k[x_1,x_2,z_1,\dots,z_{s+1-s_1},y_1,\dots,y_{M_0}]/(\tilde{f},f_2,f_3,\dots,f_{s-1}) \cong X
    $$
    Moreover, as $\tilde{f} = f + h$ in $k[x_1,x_2,z_1,z_2,z_3,y_1,\dots,y_{M_0}]$, where $f$ is as in \eqref{eq:HPT-quartic} and $h \in k[y_1,\dots,y_{M_0}]$ is some linear polynomial. We see that
    $$
        X \cong \Spec k[x_1,x_2,x_3,z_1,\dots,z_{s+1-s_1},y_1,\dots,y_{M_0}]/(f_1,\dots,f_s),
    $$
    where $f_2,\dots,f_{s_1+1}$ are as in \eqref{eq:prop:small-index-first-poly}, $f_{s_1+2},\dots,f_{s-1}$ are as in \eqref{eq:prop:small-index-linear-poly}, and
    $$
        f_1 :=  x_1 z_1^2 + x_2 z_2^2 + x_1 x_3^{3 -d_s} + (1 + x_1^2 + x_2^2 - 2x_1 - 2x_2 - 2x_1x_2) + h, \quad f_s := x_3 - x_2 z_3^{d_s -1}.
    $$
    Note that $3 + s+ 1 -s_1 + M_0 = 4 + M$ and that the leading monomials of the polynomials $f_1,\dots,f_s$ are relative prime for the graded monomial ordering
    $$
        x_1 > z_1 > x_2 > x_3 > z_2 > \dots > z_{s+1-s_1} > y_1 > y_2 > \dots > y_{M_0},
    $$
    which concludes this case.

    \textbf{Case c.} Assume $d_1 = 3$ and $M = 3s -3$, in particular $d_1 = \dots = d_s = 3$. 
    By \Cref{cor:HPT-2cubics}, there exist cubic polynomials $c_1,c_2 \in k[x_1,x_2,\dots,x_7]$ such that the assumptions \labelcref{item:thm_add-hypers:integral,item:thm_add-hypers:ratl-pt,item:thm_add-hypers:tor-order,item:thm_add-hypers:strongly-rat} of \Cref{thm:add-hypers} are satisfied.
    By applying \Cref{thm:add-hypers} \ref{item:thm_add-hypers:add-hypers} $(s-2)$-times with added degree $3$ and $3$ added variable each time, we obtain polynomials
    $$
        f_1,\dots,f_{s} \in k[x_1,\dots,x_{3s+1}]
    $$
    of degree $\deg f_i = 3$ such that the affine scheme
    $$
        X := \Spec k[x_1,\dots,x_{3s+1}]/(f_1,\dots,f_s)
    $$
    is an integral $k$-variety of dimension $2s+1$ and satisfy $\Tor^{\Z/2}(X,W) = 2$ for some non-empty closed subscheme $W \subset X$, which completes the third and final case.
\end{proof}

\begin{theorem}\label{thm:ci-torsion-order-low-index}
    Let $k$ be a field of characteristic different from $2$. Then the torsion order of a very general complete intersection $X \subset \CP^N$ of multidegree $(d_1,\dots,d_s) \in \Z^s_{\geq 2}$ is divisible by $2$ if $N \geq 4 + s$ and the Fano index $r := N +1 - d_1 -\dots - d_s \leq 2$.
\end{theorem}

\begin{remark}
    \Cref{thm:ci-torsion-order-low-index} covers the complete intersection of $s_1$ cubic hypersurface and $s_2$ quadrics in $\CP^{3s_1 + 2s_2 + 1}$, in particular the complete intersection of $s$ quadrics in $\CP^{2s + 1}$. These cases are not contained in \Cref{thm:ci-torsion-order}.
    The special case of a $(2,3)$-fourfold, originally due to Skauli \cite{Ska23}, was recently reproved by Fiammengo--Lüders \cite{FL25} using the methods of \cite{PS23,LS24}. The input from \cite{HPT18} here is replaced with the stable irrationality result of cubic threefolds due to Engel--de Gaay Fortman--Schreieder \cite{EGFS25}, which builds on \cite{Voi17}.
\end{remark}

\begin{proof}
    We prove the theorem by essentially the same argument as in \Cref{thm:ci-torsion-order}. 
    Since the torsion order and the notion of very general behave well under field extensions, we can assume that $k$ is algebraically closed and uncountable.
    By \Cref{prop:aff-ci-low-index}, there exists an affine complete intersection $X^\circ$ of multidegree $(d_1,\dots,d_s)$ in $\aff^N_k$ such that $\Tor^{\Z/2}(X^\circ,W^\circ) = 2$ for some non-empty closed subscheme $W^\circ \subset X^\circ$. 
    Moreover, the leading monomial of a set of defining equations of $X^\circ$ are relatively prime for some graded monomial ordering.
    Let $X$ be the closure of $X^\circ$ in $\CP^N$. \Cref{prop:homogenization-monomial-order} shows that $X$ is a complete intersection in $\CP^N$ of multidegree $(d_1,\dots,d_s)$ such that
    $$
        \Tor^{\Z/2}(X,W) = \Tor^{\Z/2}(X^\circ,W^\circ) = 2
    $$
    for the non-empty closed subset $W \subset X$ satisfying $X \setminus W = X^\circ \setminus W^\circ \subset X$.

    Let $X_{d_1,\dots,d_s}$ be a very general complete intersection in $\CP^N$ of multidegree $(d_1,\dots,d_s)$.
    Since the base change to an algebraically closed field extension does not affect the torsion order by \Cref{lem:properties-of-torsion-order}, we can assume that $X_{d_1,\dots,d_s}$ degenerates to $X$ by \Cref{lem:very-general} \ref{item:very-gen:degeneration}.
    By choosing $\dim X$ hyperplane sections through a closed point of $W$ in the total space of the degeneration, we can assume that there exists a non-empty zero-dimensional closed subscheme $W_{d_1,\dots,d_s} \subset X_{d_1,\dots,d_s}$ specializing to a closed subscheme of $W$. Thus, \Cref{lem:torsion-order-degeneration} implies that
    $$
        2 \mid \Tor^{\Z/2}(X_{d_1,\dots,d_s},W_{d_1,\dots,d_s}).
    $$
    In particular, we find that $2$ divides $\Tor(X_{d_1,\dots,d_s})$ by \Cref{lem:properties-of-torsion-order}, as $\Tor(X_{d_1,\dots,d_s}) = \infty$ or $\CH_0(X_{d_1,\dots,d_s}) \cong \Z$ by \Cref{lem:torsion-order-and-CH0}.
\end{proof}

\begin{proof}[Proof of \Cref{thm:intro-ci-2-torsion}:]
    The case $r \leq 2$ is exactly \Cref{thm:ci-torsion-order-low-index}. We turn to the other case.
    Let $d_i \geq 4$ be a integers and set $n := d_i - 2 \geq 2$. A small computation shows that
    $$
        2^n + \sum\limits_{j=0}^{n-1} \binom{n}{j} \left\lfloor \frac{j}{2} \right\rfloor - 2 = 2^n + (n-1)2^{n-2} -\left\lfloor \frac{n}{2} \right\rfloor - 2 = (n+3)2^{n-2} - \left\lfloor \frac{n+4}{2} \right\rfloor,
    $$
    see e.g.~\cite[Lemma 7.4]{LS24}. Thus, \Cref{thm:ci-torsion-order} with $m = 2$ implies the theorem.
\end{proof}

\subsection{Hypersurfaces in products of projective spaces}

We aim to compactify $\aff^N$ to a product of projective spaces. This leads to new bounds on the torsion order of hypersurfaces in products of projective spaces and subsequently also statements about the irrationality of such hypersurfaces.

\begin{proposition}\label{prop:aff-hypers-prod}
    Let $k$ be an uncountable algebraically closed field and let $n,m \geq 2$ and $s \geq 0$ be integers such that $m \in k^\ast$.
    Then for any $(s+1)$-tuples $$
        (M_0,M_1,\dots,M_s), \, (d_0,d_1,\dots,d_s) \in \Z_{\geq 1}^{s+1}
    $$
    satisfying $d_0 \geq n+m$, $d_i \geq M_i + 1$ for all $i \in \{1,\dots,s\}$, and
    \begin{equation}\label{eq:prop_aff-hypers-prod:dimension-bound-hypers}
        4 \leq M_0 \leq n + 2^n - 1 + \sum\limits_{j=0}^{n-1} \binom{n}{j} \left\lfloor \frac{j}{m} \right\rfloor,
    \end{equation}
    there exists an irreducible polynomial $f \in k[y_{0,1},\dots,y_{0,M_0},y_{1,1},\dots,y_{1,M_1},y_{2,1},\dots,y_{s,M_s}]$ of multidegree $(d_0,d_1,\dots,d_s)$ such that
    $$
        \Tor^{\Z/m}\left(\Spec k[y_{0,1},\dots,y_{s,M_s}]/(f), \Spec k[y_{0,1},\dots,y_{s,M_s}]/(f,\partial_{y_{0,M_0}} f)\right) = m
    $$
\end{proposition}

\begin{proof}
    Fix (non-negative) integers $n$, $m$, and $s$ as in the statement of the proposition. Let $M_0,M_1,\dots,M_s,d_0,d_1,\dots,d_s$ be positive integers satisfying $d_0 \geq n+ m$, $d_i \geq M_i + 1$ for each $i \in \{1,\dots,s\}$, and \eqref{eq:prop_aff-hypers-prod:dimension-bound-hypers}.

    Since $d_0 \geq n + m$ and $M_0$ satisfies the bounds in \eqref{eq:prop_aff-hypers-prod:dimension-bound-hypers}, \Cref{thm:base-examples-hypersurfaces} shows that there exists an irreducible polynomial $f \in k[y_{0,1},\dots,y_{0,M_0}]$ of degree $d_0$ satisfying
    $$
        \Tor^{\Z/m}\left(\Spec k[y_{0,1},\dots,y_{0,M_0}]/(f), \Spec k[y_{0,1},\dots,y_{0,M_0}]/(f,\partial_{y_{0,1}} f) \right) = m
    $$
    and condition \eqref{eq:condition-star2} with $z = y_{0,1}$.
    Hence, the proposition holds for $s = 0$ and we can assume now that $s \geq 1$. 
    Note that the conditions \labelcref{item:thm_add-hypers:integral,item:thm_add-hypers:ratl-pt,item:thm_add-hypers:strongly-rat,item:thm_add-hypers:tor-order} in \Cref{thm:add-hypers} are satisfied for $f$ and the integers $n = M_0 - 1$ and $r = 0$.

    By repeatedly applying \Cref{thm:add-hypers} \ref{item:thm_add-hypers:increase-deg} to $f$ with added degrees $d_1-1,\dots,d_s-1$ and $M_1,\dots,M_s$ added variables, we obtain an irreducible polynomial
    $$
        \check{f} \in k[y_{0,1},\dots,y_{0,M_0},y_{1,1},\dots,y_{1,M_1},y_{2,1},\dots,y_{s,M_s}]
    $$
    of multidegree $(d_0,d_1,\dots,d_s)$ such that 
    $$
        \Tor^{\Z/m}\left(\Spec k[y_{0,1},\dots,y_{s,M_s}]/(\check{f}),\Spec k[y_{0,1},\dots,y_{s,M_s}]/(\check{f},\partial_{y_{0,M_0}} \check{f})\right) = m,
    $$
    which finishes the proof of the proposition.
\end{proof}

\begin{theorem}\label{thm:torsion-hyper-in-prod}
    Let $k$ be a field and let $n,m \geq 2$ be integers. Then the torsion order of a very general hypersurface in $\CP^{M_0} \times \dots \times \CP^{M_s}$ of multidegree $$
        (d_0,d_1,\dots,d_s) \geq (n+m,M_1+1,\dots,M_s+1)
    $$ is divisible by $m$, if $$
        4 \leq M_0 \leq n + 2^n - 1 + \sum\limits_{l=0}^{n-1} \binom{n}{l} \left\lfloor \frac{l}{m} \right\rfloor.
    $$
\end{theorem}

\begin{proof}
    Since the notion of very general is stable under field extension by \Cref{lem:very-general} \ref{item:very-gen:field-ext} and the torsion order of a $k$-variety is divisible by the torsion order of its base change to a field extension, we can assume without loss of generality that the field $k$ is algebraically closed and uncountable.
    
    Let $M_0,\dots,M_s \in \Z_{\geq 1}$ be positive integers such that 
    $$
        4 \leq M_0 \leq n + 2^n - 1 + \sum\limits_{l=0}^{n-1} \binom{n}{l} \left\lfloor \frac{l}{m} \right\rfloor.
    $$
    For any $(s+1)$-tuple $(d_0,\dots,d_s) \in \Z_{\geq 1}^{s+1}$ with $d_0 \geq n +m$ and $d_i \geq M_i + 1$ for $i \in \{1,\dots,s\}$, \Cref{prop:aff-hypers-prod} shows that there exists an irreducible polynomial $$
        f^\circ \in k[y_{0,1},\dots,y_{0,M_0},y_{1,1},\dots,y_{1,M_1},y_{2,1},\dots,y_{s,M_s}]
    $$
    of multidegree $(d_0,d_1,\dots,d_s)$ such that the $\Z/m$-torsion order of the integral $k$-variety
    $$
        X^\circ := \{f^\circ = 0\} \subset \aff^{M_0} \times \dots \times \aff^{M_s}
    $$
    relative to some closed subscheme $W^\circ \subset X^\circ$ is equal to $m$.

    Let $Y := \CP^{M_0} \times \dots \times \CP^{M_s}$ be the product of projective space and view $\aff^{M_0 + \dots + M_s} \subset Y$ as a standard open subscheme.
    Let $X \subset Y$ be the closure of the locally closed subscheme $X^\circ$ inside $Y$ and let $f \in k[y_{0,0},\dots,y_{0,M_0},y_{1,0},\dots,y_{1,M_1},y_{2,0},\dots,y_{s,M_s}]$ be the (multi-)homogenization of $f^\circ$. Then it follows immediately that
    $$
        X = \{f = 0\} \subset Y = \CP^{M_0} \times \dots \times \CP^{M_s},
    $$
    is an integral hypersurface of multidegree $(d_0,d_1,\dots,d_s)$ satisfying $\Tor^{\Z/m}(X,W) = m$ for some closed non-empty subscheme $W \subset X$, in fact $W$ is the unique closed und reduced subscheme such that $X \setminus W = X^\circ \setminus W^\circ \subset X$.

    Up to a base change to an algebraically closed field extension, which does not a effect the torsion order (\Cref{lem:properties-of-torsion-order}), we can assume that a very general hypersurface $X_{d_0,\dots,d_s}$ of multidegree $(d_0,d_1,\dots,d_s)$ in $\CP^{M_0} \times \dots \times \CP^{M_s}$ degenerates to $X$, see \Cref{lem:very-general} \ref{item:very-gen:degeneration}.

    Consider the total space $\mathcal{X}$ of the degeneration, which is a hypersurface in the product of projective spaces $\CP^{M_0} \times \dots \times \CP^{M_s}$ over some discrete valuation ring.
    The intersection of $\mathcal{X}$ with $M_0 + \dots + M_s -1$ general hyperplane sections through a closed point in $W$ gives a closed subscheme of relative dimension $0$, whose intersection with $W$ is non-empty. Thus we can assume that there exists a non-empty closed subscheme $\mathcal{W} \subset \mathcal{X}$ of relative dimension $0$ such that the intersection with $X$ is contained in $W$.

    Applying \Cref{lem:torsion-order-degeneration} we find that
    \begin{equation}\label{eq:thm:hyp-in-product:relative-tor-ord}
        \Tor^{\Z/m}(X_{d_0,\dots,d_s},\mathcal{W} \cap X_{d_0,\dots,d_s}) = \Tor^{\Z/m}(X,W) = m,
    \end{equation}
    where $\mathcal{W} \cap X_{d_0,\dots,d_s} \subset X_{d_0,\dots,d_s}$ is a non-empty zero-dimensional closed subscheme.

    A simple computation using the Künneth formula shows that $H^{1,0}(X_{d_0,\dots,d_s}) = 0$, which uses that $M_0 > 1$. Thus \Cref{lem:torsion-order-and-CH0} shows that $\Tor(X_{d_0,\dots,d_s}) = \infty$ or $\CH_0(X_{d_0,\dots,d_s})= \Z$.
    In the latter case, \Cref{lem:properties-of-torsion-order} and \eqref{eq:thm:hyp-in-product:relative-tor-ord} show that the $\Tor(X_{d_0,\dots,d_s})$ is divisible by $m$.
\end{proof}

\begin{proof}[Proof of \Cref{thm:intro-hyper-in-prod}:]
    Let the assumptions be as in the statement of \Cref{thm:intro-hyper-in-prod}. We set $n := \log_2(M_1)$ and note that $n \geq \log_2(4) = 2$. Moreover, $$
        M_1 = 2^n \leq 2^n + n -1 + \sum\limits_{j=0}^{n-1} \binom{n}{j} \left\lfloor \frac{j}{m} \right\rfloor.
    $$
    Thus the statement follows directly from \Cref{thm:torsion-hyper-in-prod}.
\end{proof}

\subsection{Hypersurfaces in Grassmannians}

To illustrate the flexibility of the method, we consider also hypersurfaces in Grassmannians over $\C$. Recall briefly some standard facts of Grassmannians, see e.g.~\cite[Section 3.2]{EH16}. Let $k$ be a field and $l,n \in \Z_{\geq 1}$ be positive integers such that $l \leq n$. We denote the Grassmannian variety, whose closed points parametrize $l$-dimensional vector subspaces of $k^{n}$, by $\Grass(l,n)$. Throughout this section, we fix a Plücker embedding
$$
    \iota_{Pl} \colon \Grass(l,n) \hookrightarrow \CP_k^{\binom{n}{l}-1} =: \CP,
$$
which is a closed embedding. Let $U \subset \CP$ be a standard open affine chart of the projective space $\CP$. Then it follows from the construction of the Grassmannian that the scheme-theoretic intersection $V := U \cap \Grass(l,n) \subset \CP$ is isomorphic to $\aff^{l(n-l)}$.
In fact, $V$ is the intersection of $\dim U - \dim V$ hyperplanes in $U$ see e.g. \cite[Proposition 3.2]{EH16}.

\begin{lemma}\label{lem:hypers-in-Grass}
    Let the notation be as above. Then for any integral affine hypersurface $X^\circ \subset V \cong \aff^{l(n-l)}$ of degree $d$, there exists an integral affine hypersurface $Y^\circ \subset U \cong \aff^{\binom{n}{l}-1}$ of degree $d$ such that the scheme-theoretic closure $X$ of $X^\circ$ in $\Grass(l,n)$ is isomorphic to $$
        X \cong Y \times_{\CP} \Grass(l,n),
    $$
    where $Y \subset \CP$ denotes the scheme-theoretic closure of $Y^\circ$ in $\CP = \CP^{\binom{n}{l}-1}$.
\end{lemma}

\begin{proof}
    Since $V$ is the intersection of hyperplanes in $U$, we can write the global sections of $U$ as $k[U] = k[V][y_1,\dots,y_{\dim U - \dim V}]$. In particular, there exist morphisms of $k$-varieties
    $$
        \iota_V \colon V \hookrightarrow U \quad \text{and} \quad p_V \colon U \longrightarrow V
    $$
    such that $p_V \circ \iota_V = \id_V$. We claim that the fibre product $Y^\circ = U \times_V X^\circ$ satisfies the property claimed in the lemma. Note that $Y^\circ$ is an integral affine hypersurface of degree $d$, as $k[U] = k[V][y_1,\dots,y_{\dim U - \dim V}]$.

    Consider the morphisms
    $$
        f \colon Y^\circ \xrightarrow{p_V} X^\circ \hookrightarrow \Grass(l,n) \quad \text{and} \quad g \colon Y^\circ \hookrightarrow \CP = \CP^{\binom{n}{l}-1},
    $$ 
    where the unspecificied arrows are the natural inclusions. Then it follows from the universal property of the fibre product that the scheme-theoretic image $\Ima f$ of $f$ is isomorphic to the fibre product
    $$
        \Ima f \cong \Ima g \times_{\CP} \Grass(l,n).
     $$
     Note that $Y = \Ima g$. Thus it suffices to show that $\Ima f$ is equal to the scheme-theoretic closure $X$ of $X^\circ$ in $\Grass(l,n)$.
     
     The equality of morphisms $p_V \circ \iota_V = \id_V$ implies that as sets $f(Y^\circ)$ is equal to the set-theoretic image of $X^\circ$ in $\Grass(l,n)$ via the natural inclusion.
     Since $Y^\circ$ and $X^\circ$ are reduced schemes, we find that $\Ima f$ is equal to the scheme-theoretic closure $X$ of $X^\circ$ in $\Grass(l,n)$, see e.g.\ \cite[Remark 10.32]{GW-AG1}.
\end{proof}

The lemma together with the examples of affine hypersurfaces in \Cref{thm:base-examples-hypersurfaces} shows the following theorem about the torsion order of very general hypersurfaces in Grassmannians.

\begin{theorem}\label{thm:torsion-Grass}
    Let $n',m \geq 2$ and $l,n \geq 1$ be positive integers such that $l \leq n$ and set $N := \binom{n}{l}-1$. Fix a Plücker embedding $\Grass(l,n) \hookrightarrow \CP^{N}$ of the Grassmannian $\Grass(l,n)$ over $\C$. Then the intersection of $\Grass(l,n)$ with a very general hypersurface of degree $d \geq n' + m$ in $\CP^{N}_\C$ has torsion order divisible by $m$, if
    \begin{equation}\label{eq:thm_hypers-Grass-dim-bound}
        4 \leq \dim \Grass(l,n) = l(n-l) \leq 2^{n'} - 1 + \sum\limits_{j=0}^{n'-1} \binom{n'}{j} \left\lfloor \frac{j}{m} \right\rfloor + d - m.
    \end{equation}
\end{theorem}

\begin{proof}
    Let $n',m,l,n$ be integers as in the theorem and write $N := \binom{n}{l}-1$. Since $l(n-l)$ satisfies the bound \eqref{eq:thm_hypers-Grass-dim-bound}, there exists by \Cref{prop:aff-ci} an affine hypersurface
    $$
        X^\circ \subset \aff^{l(n-l)}
    $$
    of degree $d \geq n' + m$ such that $\Tor^{\Z/m}(X^\circ,W^\circ) = m$ for some closed subscheme $W^\circ \subset X^\circ$.

    Let $X \subset \Grass(l,n)$ and $Y \subset \CP^N$ be the $k$-varieties constructed in \Cref{lem:hypers-in-Grass}. We denote by $W \subset X$ the unique reduced closed subscheme of $X$, whose complement is equal to $X^\circ \setminus W^\circ$ in $X$.
    
    Let $Y_d$ be a very general hypersurface of degree $d$ in projective space $\CP^N$ over $k$.
    Up to replacing $\C$ by an algebraically closed field extension, we can assume by \Cref{lem:very-general} \ref{item:very-gen:degeneration} that $Y_d$ degenerates to $Y$.
    Let $X_d$ be the intersection of $Y_d$ with $\Grass(l,n)$, which is embedded in $\CP^N$ via the fixed Plücker embedding. 
    Then $X_d$ degenerates to $X$. By choosing $l(n-l)-1$ general hyperplane sections, we can assume that there exists a non-empty closed subscheme of the total space of the degeneration of relative dimension $0$ such that the intersection with $X$ is contained in $W$. We denote the intersection of that closed subscheme with $X_{d}$ by $W_d$.
    
    Then \Cref{lem:torsion-order-degeneration} implies that $\Tor^{\Z/m}(X_d,W_d) = m$. We aim to deduce that $m$ divides $\Tor(X_d)$, which could be $\infty$. By \Cref{lem:torsion-order-and-CH0} and \Cref{lem:properties-of-torsion-order}, it suffices to prove that $H^{0}(X_d,\Omega_{X_d}^1) = 0$. Consider the long exact sequence associtated to the conormal short exact sequence of the (Cartier) divisor $X_d$ in $\Grass(l,n)$
    $$
        \dots \longrightarrow H^0(\Grass(l,n),\Omega^1_{\Grass(l,n)}) \longrightarrow H^0(X_d,\Omega_{X_d}^1) \longrightarrow H^1(X_d,O_{X_d}(-d)) \longrightarrow \dots
    $$
    It is well-known that the integral cohomology ring of a Grassmannian is generated by the Schubert classes and thus only non-trivial in even degree. In particular the first term in the above exact sequence vanishes.
    The vanishing of the last term follows directly from Kodaira's vanishing theorem applied to ample line bundles of the form $\mathcal{O}_{\Grass(l,n)}(j)$ on $\Grass(l,n)$ for $j > 0$. 
    Thus we have $H^0(X_d,\Omega_{X_d}^1) =0$ by exactness, which implies that $m$ divides $\Tor(X)$ by the above argument.
\end{proof}

\begin{remark}\label{rem:hypers-in-Grass}
    The assumption on the base field being $\C$ is mostly for simplicity and the argument also works over arbitrary fields up to the computation of $h^{1,0}(X_d) = 0$.
\end{remark}

\begin{proof}[Proof of \Cref{thm:intro-2-torsion-Grass}:]
    The statement follows from \Cref{thm:torsion-Grass} together with the equality
    $$
        \sum\limits_{j=0}^{n'-1} \binom{n'}{j} \left\lfloor \frac{j}{2} \right\rfloor = (n'-1) 2^{n'-2} - \left\lfloor \frac{n'}{2} \right\rfloor,
    $$
    see e.g.~\cite[Lemma 7.4]{LS24}. 
\end{proof}

\end{document}